\documentclass[preprint,10pt,twoside]{amsart}
\usepackage[latin1]{inputenc}
\usepackage[dvips]{graphics}
\input xy
\input cyracc.def

\xyoption{all}
\usepackage{xr}
\usepackage{amssymb}
\usepackage{amsmath}
\usepackage{amsthm}
\usepackage{amsfonts}
\usepackage{mathrsfs}
\theoremstyle{plain}
\newtheorem{T}{Theorem}[section]

\newtheorem{Cor}[T]{Corollary}
\newtheorem{TL}[T]{Lemma}
\newtheorem{Prop}[T]{Proposition}

\newcommand\mps[1]{\marginpar{\small\sf }}

\theoremstyle{definition}\newtheorem{D}{Definition}
\theoremstyle{remark}\newtheorem{Ex}{Example}
\theoremstyle{remark}
\theoremstyle{remark}\newtheorem{remark}{Remark}[T]

\newcommand{\Aut}{\hbox{Aut}}

\newcommand{\bb}[1]{\mathbb{#1}}
\newcommand\calo{\mathcal O}

\newcommand\ibid{\emph{ibid}}
\newcommand\idem{\emph{idem}}
\newcommand\loccit{\emph{loc.cit.}}
\newcommand{\mc}[1]{\mathcal{#1}}

\newcommand{\lci}{local complete intersection }

\newcommand{\Hom}{\mathrm{Hom}}

\newcommand\op[1]{\operatorname{#1}}

\newcommand\rk{\operatorname{rk}}

\input cyracc.def

\theoremstyle{plain}

\theoremstyle{definition}

\newcommand{\ses}[5]{$$
\xymatrix@1{ 0 \ar[r] & {{#1}} \ar[r]^-{{#2}} & {{#3}} \ar[r]^-{{#4}} &
{{#5}} \ar[r] & 0
\\ }$$}

\newcommand{\sesdot}[5]{$$
\xymatrix@1{ 0 \ar[r] & {{#1}} \ar[r]^-{{#2}} & {{#3}} \ar[r]^-{{#4}} &
http://www.expressen.se/nyheter/1.2165248/braket-i-radiodebatten-ohly-halsar-inte
{{#5}} \ar[r] & 0. \\ }$$}

\newcommand{\sesbig}[5]{$$\xymatrix@1{{\raisebox{1.0ex}[3.0ex][1.0ex]{$0$}}
\ar@<0.6ex>[r] & {\raisebox{1.0ex}[3.0ex][1.0ex]{${#1}$}}
\ar@<0.6ex>[r]^{#2} & {\raisebox{1.0ex}[3.0ex][1.0ex]{${#3}$}}
\ar@<0.6ex>[r]^{#4} & {\raisebox{1.0ex}[3.0ex][1.0ex]{${#5}$}}
\ar@<0.6ex>[r] & {\raisebox{1.0ex}[3.0ex][1.0ex]{$0$}}}$$}

\newcommand{\Comment}[1]{}

\title{The excess formula in functorial form}
\author{Dennis Eriksson}
\address{Mathematical Sciences, University of Gothenburg and Chalmers University Of Technology, 41296 Gothenburg, Sweden}
\email{dener@chalmers.se}
\begin{document}
\maketitle
\date{\today}
\begin{abstract} This article is motivated by the need for better understanding of refined Riemann-Roch theorems and the behavior of the determinant of the cohomology. This poses a certain problem of functoriality and can be understood as that of giving refined constructions of operations in algebraic $K$-theory. In this article this is specialized to mean refining the excess formula, which measures the failure of base change, to the level of Deligne's virtual category. We give a natural set of properties for such a refinement, and prove that there exists a unique family of excess formulas on this refined level satisfying these properties. 
\end{abstract}
\tableofcontents
\newpage

\section{Introduction} Recall the excess formula in $K$-theory. It relates the behavior of base change of pushforward whenever the resulting diagram is not Tor-independent (for example non-flat base change) and we use the following formulation: For a Cartesian diagram $$\xymatrix{ X \ar[r]^{g'} \ar[d]^{f'} & Y \ar[d]^f \\
X' \ar[r]^g & Y',}$$ where $g$ and $g'$ are arbitrary morphisms, and $f$ and $f'$ are projective \lci morphisms we have an equality
 $$f'_!\left( \lambda_{-1}(E) \otimes {g'}^!(x) \right) =  g^! f_!(x)$$
 in $K_0(X')$, for $x \in K_0(Y)$, at least if the schemes are suitable (having the resolution property, see Section \ref{chapter:excess}). Here $E$ denotes the excess bundle measuring the difference of the conormal bundles of $f$ and $f'$ when they are closed immersions (see Section 3, Definition \ref{defn:excessbundle}). A reference is \cite{Riemann-RochAlgebra}, chapter VI. The functorial form of the excess formula is then a construction of a canonical isomorphism of objects in the virtual category, a certain categorical refinement of $K_0$ introduced in \cite{determinant} (see below and beginning of section \ref{chapter:excess}), that realizes this formula. Informally, the main issue of this article is to show that the deformation to the normal cone is "functorial enough", which forces us to go through a lot of cumbersome notation, and to give a set of properties that characterize the most obvious candidate for an excess isomorphism which is obtained by the aforementioned deformation. We solve this problem and give properties that characteristize the lifting uniquely, up to sign. \\\\
The general framework is inspired by Deligne's article \cite{determinant} where there is a proposed problem of constructing refinements in terms of "higher equivalences", of the Grothendieck-Riemann-Roch identity $$\op{ch}(f_! V) = f_* (\op{ch}(V) \op{Td}(T_f)),$$ which is valid in Chow-theory, for a proper morphism $f: X \to Y$ of suitable schemes and $V$ is a vector bundle on $X$. Here $f_!$ (resp. $f_*$) is the pushforward in $K$-theory (resp. Chow-theory), see \cite{Riemann-RochAlgebra}, chapter V, \S 7 or \cite{SGA6}, VIII, Th\'eor\`eme 3.6 for a precise formulation of the theorem. We insist on the notation $f_!$ to distinguish it from the direct image of a sheaf. The term "higher equivalences" can be interpreted as introducing functoriality, i.e. replacing the groups involved by categories, the maps by functors and the identity by a natural transformation of functors. In \cite{determinant} it is suggested that there might be a higher categorical framework for this. As an approximate categorification of the $K$-theory involved, Deligne introduces the virtual category, which we can think of as a truncation of some still unknown higher category. It amounts to the fundamental groupoid of the Quillen $Q$-construction of the category of vector bundles on the scheme in question, which has the structure of a category and satisfies a certain universal property similar to that of $K_0$ (see beginning of section \ref{chapter:excess}). It is particulary well-suited for studying secondary invariants in $K$-theory and the degree one part of the above identity, and can be seen to contain information about the determinant of the cohomology. This is also what is studied in \cite{determinant}, in the case of families of smooth curves, where the best possible results are obtained, completely describing the twelfth power of the determinant of the cohomology in terms of functorial line bundles corresponding to terms in Grothendieck-Riemann-Roch (see \loccit~Th\'eor\`eme 9.9 for a precise statement). For degenerations of curves this isomorphism was studied by T. Saito and was used to prove a discriminant-conductor-formula (cf. \cite{T.Saito-conductor}).\\\\
One of issues encountered when studying the above types of functoriality is related to the excess formula, the main topic of this article. In \cite{DRR3} this formula will serve as the model for the functorial Grothendieck-Riemann-Roch-theorem in the case of closed immersions, but should have independent interest. The excess formula is a case where one can make use of the deformation to the normal cone which is general enough to go through the main arguments of functorality, and at the same time an easy enough situation so that the arguments become writable. It should also be noted that these types of excess formulas are one of the main ingredients for Lefschetz fixed point formulas in equivariant $K$-theory (cf. \cite{Thomason7}, \cite{Thomason5}) whose functorial versions we will also return to in the future article already mentioned. The main result (Theorem \ref{thm:excess}) is also a $K$-theoretic version of a result by Franke in \cite{Franke-unpublished}, which unfortunately remains unpublished, where a functorial form of the Grothendieck-Riemann-Roch theorem already appears. It is moreover a slight improvement on Theorem 3.1 in \cite{Thomason5}, which gives the theorem "up to homotopy". The present article ameliorates this to give the same type of result but refined "up to homotopy up to homotopy" (and "up to sign"), when we consider the virtual category as a truncation of the $K$-theory space using the description as a fundamental groupoid mentioned above. \\\\
The article is organized as follows. In Section 2 and 3 we establish preliminaries for virtual categories and define a "rough" excess isomorphism. "Rough" here indicates that it a priori depends on some choices in a model situation. In Section 4 we formulate the main result (Theorem \ref{thm:excess}), which states that given a natural set of properties of the excess isomorphism, it exists and is uniquely determined by those properties. This is proved in the sections 5-7. The proof of the main result proceeds by a deformation to the normal cone argument which we review in the appendix together with some more or less well-known facts, formulated in the category of algebraic stacks. \\\\
It should finally be noted that one of the main initial  motivations for Deligne's program was to understand the Quillen metric. From this point of view the deformation to the normal cone is quite natural, since it is usually not enough to compute with Tor-functors as in \cite{SGA6} loc.cit., if one also wants to control metrics. This particular application to metrics will be considered in another paper. \\\\
Acknowledgements: I'd like to thank Marc Levine, Damian R\"ossler and Takeshi Saito for several helpful commentaries and criticisms on an early version of the material that is in this paper. I also want to thank Jo\"el Riou for pointing out that an argument in section \ref{section:generalexcess} was insufficient and Lars Halvard Halle for his reading and comments on the manuscript.
\newpage

\section{Preliminaries and conventions} \label{chapter:excess} \begin{flushright}  \end{flushright}

In this article all schemes and algebraic spaces are separated and noetherian. A(n algebraic) stack is an Artin stack $X$ over a basescheme $S$. Given a smooth presentation $X_0 \to X$ it is equivalent to (cf. \cite{MB}, Proposition 4.3.1) an $S$-space in groupoids $[X_\bullet] = [X_1 {\rightrightarrows}_t^s X_0]$ (see \idem, 2.4.3, for the definition), where $X_i$ are $S$-algebraic spaces and the two morphisms $s$ and $t$ are smooth and \linebreak $\delta = (s,b): X_1 \to X_0 \times_S X_0$ is separated and quasi-compact. A morphism of stacks is a 1-morphism, and we will pretend that there are no issues with 2-morphisms. For the reader worried about this the results only apply to the case of equivariant schemes (possibly with trivial group), which is also the main application in mind. These problems can be circumvented by introducing the appropriate natural transformations corresponding to 2-commutative diagrams. In the same vein, a Cartesian diagram of morphisms of stacks is a 2-Cartesian diagram. A quasi-coherent sheaf $\mc F$ on $X$ is a Cartesian quasi-coherent sheaf on the associated simplicial $S$-space $X_\bullet$ for some smooth presentation $X_0 \to X$ (see \cite{Olsson-sheavesartin} Definition 6.9 for the precise definition or \cite{MB} D\'efinition 13.1.7 and Proposition 13.2.1). Morphisms of quasi-coherent sheaves are given by morphisms respecting this extra structure. A representable morphism of stacks $f: X \to Y$ gives rise to a commutative diagram of Cartesian squares $$\xymatrix{
X_1 \times_{X_0} X_1 \ar@<1ex>[d] \ar@<2ex>[d] \ar[d] \ar[r]^{f_2}& Y_1 \times_{Y_0} Y_1 \ar@<1ex>[d] \ar@<2ex>[d] \ar[d] \\
X_1 \ar[r]^{f_1} \ar@<1ex>[d] \ar[d] & Y_1 \ar@<1ex>[d] \ar[d] \\
X_0 \ar[r]^{f_0} & Y_0 }$$
where $f_i$ are morphisms of algebraic spaces, and $X_0 \to X$ is a smooth presentation and $Y_0 = Y_0 \times_Y X \to X$. We will often implicitly use this diagram to make constructions by a Cech argument. In this article we will basically only be concerned with representable morphisms. As a guiding word, properties or operations on schemes which commute with smooth base change or are local for smooth morphisms on algebraic spaces carry over to this situation, and we will assume throughout that this is so unless explicitly mentioned. We will moreover freely use the results of \cite{MB}. \\

We review also the definition of the virtual category $V(\mc C)$ of an exact category $\mc C$. Recall first that a Picard category $P$ is a symmetric monoidal groupoid such that for any object $x$ in the category, $x + : P \mapsto x + P$ is an equivalence of categories. More informally a "categorical group", or a group object in the category of groupoids, with sum and associativity-isomorphisms instead of identities (cf. \cite{determinant}, 4.1. Section 4 of \idem~is also the main reference for virtual categories used here). An additive functor of Picard categories is a functor respecting these structures.
\begin{D} [\cite{determinant}, 4.3]\label{def:determinantfunc} Denote by $(\mc C, iso)$ the subcategory of $\mc C$ with the same objects but only isomorphisms as morphisms. Suppose we have functor $$[-]: (\mc C, iso) \to P$$ where $P$ is a Picard category, satisfying the following three conditions:
\begin{enumerate}
  \item Additivity on exact sequences, i.e. for an exact sequence $$0 \to A \to B \to C \to 0$$ we have an isomorphism $[B] \simeq [A] + [C]$, and compatibility with isomorphisms of exact sequences.
  \item For a zero-object $0_{\mc C}$ (resp. $0_P$) of $\mc C$ (resp. $P$), we are given an isomorphism $[0_{\mc C}] \simeq 0_P$ in $P$. If $$0 \to A \to B \to 0 \to 0$$ is an exact sequence, with the first map being an isomorphism $f$, then the induced isomorphism $[B] \simeq [A] + [0] \simeq [A]$ is $[f]$. We demand the analogous statement for an exact sequence $0 \to 0 \to A \to B \to 0$.
    \item The additivity on exact sequences is compatible with admissible filtrations.
\end{enumerate}
Any such functor is called a determinant functor. The virtual category is a Picard category $V(\mc C)$, together with a determinant functor $[-]:(\mc C, iso) \to V(\mc C)$ such that any determinant functor $[-]: (\mc C, iso) \to P$ factors uniquely up to unique natural transformation through $(\mc C, iso) \to V(\mc C)$. \end{D}

\begin{T}[\cite{determinant}, section 4] \label{Thm:deligne-virtual} The virtual category exists, and moreover, for any Picard category $P$, the functor from $\Hom(V(\mc C), P) \to \Hom((\mc C, iso), P)$ is an equivalence of categories on the full subcategories of additive functors and determinant functors.
\end{T}

By \idem, the category $V(\mc C)$ also has the description as the fundamental groupoid of the Quillen Q-construction. Thus the group of isomorphism classes of any object of the virtual category is usual Grothendieck group, $K_0(\mc C)$ and the automorphism group of any object is isomorphic to Quillen's $K_1(\mc C)$ so the virtual category interpolates between the two. From this description it follows that faithfulness (resp. fullness) of functors between virtual categories can be read from injectivity (resp. bijectivity) on induced maps on $K_1$. An additive functor between virtual categories is an equivalence of categories if it induces an isomorphism on $K_0$ and $K_1$.

\begin{D} For an algebraic stack $X$ we denote by $V(X)$ the virtual category of vector bundles on $X$. We denote by $+: V(X) \times V(X) \to V(X)$ the natural direct sum functor of virtual vector bundles. We denote by $\otimes: V(X) \times V(X) \to V(X)$ the natural tensor product on the virtual category induced by the tensor product of vector bundles. \end{D}
Given any morphism $f: X \to Y$, we have functors $f^!: V(Y) \to V(X)$ induced by pulling back vector bundles. We will sometimes, for a vector bundle $V$, write $f^* V$ for the associated object in the virtual category to emphasize that the object is just the pullback of the vector bundle along $f$. The association $X \to V(X)$ is a functor from the 2-category of stacks to the 2-category of Picard categories (Picard categories, additive functors, natural transformation of additive functors). \\

The following is proved as in \cite{Thomason4}, Theorem 3.1, which is by reduction by naturality to the case considered in \cite{Quillen}, Theorem 2.1, p. 58.
\begin{TL} \label{lemma:projbundleformulastacks} [Projective bundle theorem] Let $X$ be an algebraic stack and $V$ be a vector bundle on $X$ of rank $r$, and set $\bb P(V) = \op{Proj}(\op{Sym} V^\vee)$. Then there is an isomorphism $K_i(X)^r \simeq K_i(\bb P(V))$ induced by $(V_i) \mapsto \sum f^* V_i \otimes \calo(i)$ on the level of vector bundles.
\end{TL}

\begin{D} \label{defn:resolutionprop} An algebraic stack $X$ has the resolution property if every coherent sheaf on $X$ admits a surjection from a vector bundle. \end{D}
The main theorem is formulated in terms of algebraic stacks (with the resolution property), but in practical situations it seems one only needs to consider the case of schemes with the action of an algebraic group. Indeed, the result in \cite{Totaro} states that noetherian normal algebraic stacks with affine stabilizer groups at closed points have the resolution property if and only if they are the stack-quotients of a quasi-affine scheme with the action of $GL_n$, and the reader not familiar with stacks lose little or nothing in supposing that we are actually even working with schemes. We refer to \cite{Thomason3} for a detailed study of schemes with the resolution property for equivariant coherent sheaves. We will need the following two statements:
\begin{TL} Suppose $X$ is an algebraic stack which has the resolution property. Then:
\begin{itemize}
  \item If $i: Z \to X$ is a (representable) closed immersion, then $Z$ has the resolution property.
  \item If $V$ is a vector bundle on $X$, then $\bb P(V)$ has the resolution property.
\end{itemize}
\end{TL}
\begin{proof} For the first point, if $\mc F$ is a coherent sheaf on $Z$, then there exists a surjection $E \to i_* \mc F$ for a vector bundle $E$ on $X$, and thus a surjection $i^* E \to i^* i_* \mc F = \mc F$. For the second point, denote by $p: \bb P(V) \to X$ the projection. Given a coherent sheaf on $\mc F$ on $\bb P(V)$, since $X$ is noetherian, for big enough $n$, there is a surjection $p^* p_* (\mc F(n)) \to \mc F(n)$ where $\mc F(n) = \mc F \otimes \calo(n)$. We conclude since $p_* (\mc F(n))$ admits a surjection from a vector bundle.
\end{proof}
The following is standard:
\begin{Cor} Suppose that $f: X \to Y$ is a (representable) projective \lci morphism of algebraic stacks, such that $Y$ has the resolution property. Then there is a direct image functor $f_!: V(X) \to V(Y)$. Moreover, given the composition of two projective \lci morphisms $f: X \to Y, g: Y \to Z$, such that $Z$ has the resolution property, then there is a natural transformation $(fg)_! = f_! g_!$.
\end{Cor}
\begin{proof} [sketch of proof] The morphism $f$ admits a factorization $X \stackrel{i}\to \bb P(V) \stackrel{p}\to Y$. By the above, all involved objects have the resolution property. We construct a direct image for $i$ and $p$ respectively and leave it to the reader to verify that the composition is independent of factorization, up to canonical natural transformation (this can be done as in \cite{Riemann-RochAlgebra}, chapter V, Proposition 5.1, or what is basically the same, Section \ref{section:generalexcess} of this article). Suppose first that $f$ is a regular closed immersion. Then the direct image $f_*$ is an exact functor from the category of vector bundles on $X$ to the category of coherent sheaves on $Y$ admitting a finite resolution of vector bundles, $P$, and induces an additive functor on the associated virtual categories. By \cite{Quillen}, Corollary 1, p. 25, the natural map $K_i(Y) \to K_i(P)$ is an isomorphism. In particular for $i=0,1$, so the induced functor on virtual categories $V(Y) \to V((P, iso))$ is an equivalence of categories and we define the functor using this. Suppose now that $f$ is a projective bundle projection. By Lemma \ref{lemma:projbundleformulastacks}, there is an equivalence of categories $V(X)^{r} \to V(\bb P(V))$ induced by $(V_i)_{i=0}^{r-1} \mapsto \sum f^* V_i \otimes \calo(i)$ and we define the direct image component wise by $f_! (f^* V_i \otimes \calo(i))= V_i \otimes \op{Sym}^i V^\vee$. This defines the required functor.
\end{proof}
The following statement includes the statement that the projection formula commutes with base change on the level of schemes and algebraic spaces.
\begin{TL} \label{lemma:projformula} Suppose that $f: X \to Y$ is a (representable) \lci projective morphism and that $Y$ has the resolution property. Then there is a projection formula isomorphism
$$f_! (f^! x \otimes y) \simeq x \otimes f_! y.$$
This isomorphism also commutes with Tor-independent base change (see next paragraph) and the composition of two projection formula isomorphisms is the projection formula isomorphism for the composition.
\end{TL}
For the next statement, recall that Tor-independence, or (cohomological) transversality, of two morphisms between schemes, from $X'$ and $X''$ to $X$ means that the Tor-sheaves $\op{Tor}_i^{\calo_{X}}(\calo_{X'}, \calo_{X''})$ vanish for $i \geq 1$. This is an \'etale local statement and translates to algebraic spaces, and the same condition for a Cartesian diagram of representable morphisms of algebraic stacks has to be interpreted as Tor-independence on a presentation $X_0 \to X$ and the induced morphisms ${X'}_0 = X_0 \times_X X' \to X_0$ and ${X''}_0 = X_0 \times_X {X''} \to X_0$. In this case the construction of the usual base change-isomorphism follows that of the case of algebraic spaces in \cite{MB}, Proposition 13.1.9 (i.e. is computed on a Cech cover), and we obtain:
\begin{TL} \label{lemma:basechange} Suppose that $$\xymatrix{ X \ar[r]^{g'} \ar[d]^{f'} & Y \ar[d]^f \\
X' \ar[r]^g & Y'}$$
is a Tor-independent Cartesian diagram of representable morphisms of stacks with the resolution property, such that $f$ and $f'$ are \lci projective morphisms. Then there is a base change isomorphism $$g^! f_! x \simeq {f'}_! {g'}^! x.$$
The composition of base change isomorphisms is the base change isomorphism of the composition.
\end{TL}
Finally, we will moreover work with a strictly commutative version of the virtual category of vector bundles, defined as follows: \begin{D} \label{Defn:strictlycomm} The objects of $V_{\pm}(X)$ and $V(X)$ are the same. We define \linebreak $\Hom_{V_{\pm}(X)}(A,B)$ as the quotient of $\Hom_{V(X)}(A,B)$ by the relation that two morphisms $h,h':A \to B$ are equal if $h \circ {h'}^{-1} = [-1] \in \Aut{(A)}$. Here, for $A = \calo$, $[-1]$ corresponds to the automorphism of the trivial vector bundle sending the unit section 1 to -1.\end{D}
\begin{remark} It is not difficult to see (\cite{determinant}, 4.9) that the above element is the automorphism of $V$ deduced from the symmetry $V + V' \simeq V + V'$ evaluated at $V=V'$ and then subtracting $V$. So commutativity becomes strict. We also recall (loc. cit. 4.11.a) that one of the main issues with signs is the following: For a distributive functor $\bigodot: V(X) \times V(X) \to V(X)$, there are two ways of constructing the isomorphism $(-V) \odot (-V') \simeq V \odot V'$. These two isomorphisms are only equal up to sign, but become completely canonical in the above category. \end{remark} We give an important example, for this article, of an operation which becomes independent of such choices, in this context.
\begin{Ex} \label{example:lambda-1} Given a vector bundle $E$, put $\lambda_{-1}(E) := \sum (-1)^i \Lambda^i E$, for exterior powers $\Lambda^\bullet$, considered as an object in the virtual category. Then the isomorphism $\lambda_{-1}(E'+E'') \simeq \lambda_{-1}(E') \otimes \lambda_{-1}(E'')$, deduced from $\Lambda^n (E'+E'') \simeq \sum \Lambda^q E' \otimes \Lambda^{n-q} E''$ becomes completely canonical and commutative in the above category. Slightly more generally, for an exact sequence, $0 \to E' \to E \to E'' \to 0$, there is a canonical isomorphism
\begin{equation} \label{lambda1} \lambda_{-1}(E) \simeq \lambda_{-1}(E') \otimes \lambda_{-1}(E'').
\end{equation}
The hidden isomorphism $\Lambda^n (E) \simeq \sum \Lambda^q E' \otimes \Lambda^{n-q} E''$ is the one coming from the filtration $$F^q \Lambda^n E = \hbox{Im} [\Lambda^q E' \otimes \Lambda^{n-q} E \to \Lambda^n E]$$ on $\Lambda^n E$, and the isomorphism $F^q \Lambda^n E/F^{q+1} \Lambda^n E\simeq \Lambda^q {E'} \otimes  \Lambda^{n-q} {E''}$. One verifies that this is also compatible with admissible filtrations. This isomorphism will be used repeatedly in what follows suit. If we could write down the isomorphism (\ref{lambda1}) in a functorial way with sign the main result of this article could probably be refined to the virtual category. \end{Ex}

\section{A rough excess isomorphism} \label{section:roughexcess}
We first give an explicit construction of an excess isomorphism in the special case that we dispose of compatible Koszul resolutions and the general problem in the general situation will be to reduce to this case. We also suppose that all the stacks have the resolution property. \\

Suppose we are given a Cartesian diagram of representable morphisms,
$$\xymatrix{ X \ar[r]^{g'} \ar[d]^{f'} & Y \ar[d]^f \\
X' \ar[r]^g & Y'}$$
with $f$ (and thus $f'$) is a closed immersion. Let $x$ be a virtual vector bundle on $Y$. Suppose we have
two morphisms $\sigma: N^\vee \to \calo_{Y'}, \sigma':
{N'}^\vee \to \calo_{X'}$ defining Koszul-resolutions of $\calo_Y$ and $\calo_{X}$ respectively that are compatible in the sense that we have a morphism $\gamma: g^* N^\vee \to {N'}^\vee$ compatible with the resolutions $g^* \sigma$ and $\sigma'$ in such a way that the natural diagram
$$\xymatrix{g^* N^\vee \ar[r]^{\gamma}  \ar[d] & {N'}^\vee \ar[d] \\
g^* \calo_{Y'} \ar[r] & \calo_{X'}}$$
commutes. This implies that $f$ and $f'$ are both closed regular immersions. \\
Writing $\mc I$ and (resp. $\mc I'$) for the ideals defining the immersions $f: X \to Y$ (resp. $f': X'\to Y'$) and let
$N_f^\vee = \mc I/\mc I^2$ and $N_{f'}^\vee = \mc I'/\mc {I'}^2$ be the conormal bundles of the immersions. By restricting to $X$ via $f'$
we obtain a commutative diagram
$$\xymatrix{{f'}^* g^* N^\vee \ar[r]^{{f'}^* \gamma} \ar[d] & {f'}^* {N'}^\vee \ar[d] \\
{g'}^* N_f^\vee \ar[r] & N_{f'}^\vee} $$
where the vertical morphisms necessarily are isomorphisms. Denote the kernel of $\gamma$ by $\bb E$.
\begin{D} \label{defn:excessbundle}We define the excess bundle $E := \ker [N_f^\vee \to N_{f'}^\vee]$ (cf. \cite{intersection}, Section 6.3. Our definition is however dual to the one given there).
\end{D}
Then ${f'}^* \bb E \simeq E$ so $\bb E$ provides an extension of $E$ to $X'$.  Also, suppose that the virtual vector bundle $x$
extends to a virtual vector bundle $x_{Y'}$ on $Y'$, i.e. there is an isomorphism $r: {f}^*
x_{Y'} \to x$. Then we define an isomorphism by
$$\xymatrix{ \Psi_{x_{Y'}, \sigma, \sigma', r}(x):  g^! f_!(x) & \stackrel{\sigma, x_{Y'}, r}\simeq & \lambda_{-1}(g^* N^\vee) \otimes g^! (x_{Y'})  \\
& \stackrel{\gamma}\simeq &  \lambda_{-1}(\bb E) \otimes  \lambda_{-1}({N'}^\vee) \otimes  g^! (x_{Y'}) \\
& \stackrel{\sigma'}\simeq &  f'_!(\calo_X) \otimes  \lambda_{-1}(\bb E) \otimes  g^! (x_{Y'}) \\
& \simeq & f'_!({f'}^!(\lambda_{-1}(\bb E) \otimes g^! (x_{Y'})))\\
& \simeq & f'_!(\lambda_{-1}(E) \otimes {g'}^! x).}$$ The first
isomorphism is given by the extension $x_{Y'}$ together with the
resolution $\sigma$ and then applying $g^!$. On the second to last
line we use the projection formula once again, and then the
definition of ${f'}^!$ for the last line. We will show that
this isomorphism is independent of the subscripts $x_{Y'}, \sigma, \sigma', r$. Actually, assuming additivity in the data, the obvious comparison immediately gives that the isomorphism does not depend on the extension.
\begin{D} The above isomorphism is the rough excess isomorphism with respect to $x_{Y'}, \sigma, \sigma', r$.
\end{D}

\begin{TL} \label{lemma:closed-base change} Suppose that the above square is Tor-independent. Then the constructed rough excess isomorphism coincides with the base change isomorphism.
\end{TL}

\begin{proof} Denote by ${g}^! f_! x \stackrel{c_x} \to f'_! {g'}^! x$ the base change isomorphism.  The statement of the lemma is that for an extension $x_{Y'}$ of $x$ to $Y'$,  the outer contour of following diagram
$$\xymatrix{g^! f_! (x) \ar[r]^-\sigma \ar[dr] \ar[dd]^{c_x} & g^! x_{Y'} \otimes g^! \lambda_{-1} N^\vee \ar[r]^\gamma \ar[d]^{\sigma} & g^! x_{Y'} \otimes \lambda_{-1} {N'}^\vee \ar@/{}^{6pc}/[dd]  \ar[d]^{\sigma'} & \\
& g^! x_{Y'} \otimes {g}^! f_! \calo_Y \ar[r]^{g^! x_{Y'} \otimes c_{\calo_Y}} & g^! x_{Y'} \otimes f'_! {g'}^! \calo_Y  \ar[d]  &   \\
f'_! {g'}^! x  & &   \ar[ll] f'_! ({f'}^! {g}^! x_{Y'}) }$$
is commutative. Here the outer right contour connecting ${g}^! f_! x$ and $f'_! {g'}^! x$ is the proposed rough excess isomorphism. First of all, the triangles commute per definition. The middle horizontal morphism is the base change isomorphism for the trivial bundle $\calo_X$ tensored with the virtual bundle ${g^! x_{Y'}}$. The square commutes since it is induced by the natural isomorphisms
$g^! \lambda_{-1} N^\vee \to \lambda_{-1} {N'}^\vee$ and the commutative diagram
$$\xymatrix{g^* \mc I \ar[r] \ar[d] & \mc I' \ar[d] \\
g^* \calo_{Y'} \ar[r] \ar[d] & \calo_{X'} \ar[d] \\
g^* f_* \calo_Y \ar[r] & {f'}_* \calo_{X} \ar@{=}[r] & {f'}_* {g'}^* \calo_Y  }$$
with exact columns and vertical isomorphisms. To show the large inner contour commutes, to identify the additive functors we can assume by Theorem \ref{Thm:deligne-virtual} that $x = f^* V'$ for an actual vector bundle $V'$ on $Y'$. Then the statement is simply compatibility of base change with the projection formula. \end{proof}

\section{Formulation of the main theorem}
Let henceforth "virtual category" and "virtual objects" be substitutes for "strictly commutative virtual category" and "strictly commutative virtual objects" (cf. Definition \ref{Defn:strictlycomm}). The main issue is that we want to apply Example \ref{example:lambda-1} above. Also make the assumption that all the algebraic stacks considered have the resolution property (cf. Section \ref{chapter:excess}, Definition \ref{defn:resolutionprop}). We will also only consider representable morphisms except when considering the deformation to the normal cone where a certain morphism to $\bb P^1$ will in general not be representable. The main result of this article is the following:

\begin{T} \label{thm:excess} Suppose we have a Cartesian square $\mc E$ of $S$-stacks (for some scheme $S$) and representable morphisms
$$\xymatrix{ X \ar[r]^{g'} \ar[d]^{f'} & Y \ar[d]^f \\
X' \ar[r]^g & Y',}$$
that $g$ and $g'$ are arbitrary morphisms and that suppose $f$ and $f'$ are projective \lci morphisms. Then there is a family of natural transformations of additive functors in the virtual category of vector bundles:

$$\Psi_{\mc E}: f'_!\left( \lambda_{-1}(E) \otimes {g'}^!(x) \right) \simeq  g^! f_!(x)$$
where $E$ is the excess bundle. It is unique if we impose the following properties (or conditions):
\begin{enumerate}
    \item \textsc{Stability under Tor-independent base change:}  Consider the following cubical diagram
$$   \xymatrix@!{
 & \widetilde X \ar@{->}[rr]^{\widetilde {g'}} \ar@{->}'[d][dd]^{q''}
   & & \widetilde Y \ar@{->}[dd]^{q}
\\
 \widetilde{X'} \ar@{<-}[ur]^{\widetilde f'} \ar@{->}[rr]^(0.40){\widetilde g}\ar@{->}[dd]^{q'''}
 & & \widetilde{Y'} \ar@{<-}[ur]^{\widetilde f}\ar@{->}[dd]^(0.35){q'}
\\
 & X \ar@{->}'[r][rr]^{g' }
   & & Y
\\
 X' \ar@{->}[rr]^{ g}\ar@{<-}[ur]^{f' }
 & & Y' \ar@{<-}[ur]^{f}
}$$ where the right-hand and left-hand vertical squares are Tor-independent and Cartesian. The upper and lower diagrams are denoted by $\tilde {\mc E}$ and $\mc E$ respectively and are supposed to be as in the introduction. Symbolically we summarize the cube by a morphism of diagrams $Q: \tilde {\mc E} \to \mc E$. Then there is a commutative diagram of natural transformations of additive functors:
$$\xymatrix{\tilde{f}'_!(\lambda_{-1}(\tilde E) \otimes \tilde {g'}^!q^! ) \ar[d]^{q \tilde{g'} = {g'} q''', {q''}^*E = \tilde E} \ar[rr]^{\Psi_{\widetilde{ \mc E}}(q^!)} & & \tilde i^! \tilde f_!q^! \ar[d]^{base~change} \\
\tilde{f}'_!(\lambda_{-1}({q''}^* E) \otimes ({q''}^! {g'}^!)) \ar[d] & & \tilde g^! {q'}^! f_! \ar[dd]^{i q''' = q' \widetilde {g}} \\
\tilde{f}'_! {q''}^! (\lambda_{-1}(E) \ar[d]^{base change} \otimes {g'}^!))  \\
{q'''}^!(f_!(\lambda_{-1}(E)  \otimes {g'}^!))
\ar[rr]^{{q'''}^*\Psi_{\mc E} } & & {q'''}^! g^! f_!.}$$ We
will sometimes write this as $\Psi_{\widetilde {\mc E}} \circ Q^*
\simeq {Q}^* \circ \Psi_{\mc E}.$ In particular there is an isomorphism $\Psi_{\widetilde{ \mc E}}(q^*(x)) \simeq
{q'''}^*\Psi_{\mc E} (x)$.
    \item \textsc{Stability under the projection formula:} There is a commutative diagram
    $$\xymatrix{f'_!\left( \lambda_{-1}(E) \otimes {g'}^!(x \otimes f^! y) \right) \ar[r] \ar[d] & g^! f_!(x \otimes f^! y) \ar[d] \\
    f'_!\left( \lambda_{-1}(E) \otimes {g'}^!(x) \right) \otimes {g}^! y  \ar[r]  & g^! f_!(x)  \otimes g^! y }$$
    where the horizontal isomorphisms are given by excess and the vertical ones are given by the projection formula. \\
    \item \textsc{Normalization:}  Suppose that $f$ is a closed embedding of a Cartier divisor $Y$ in $Y'$, and that $X = X'$. Let
     $$0 \to \calo(-Y) \to \calo_{Y'} \stackrel{\sigma}\longrightarrow \calo_Y \to 0$$
     be the canonical Koszul resolution. Then, whenever $x$ extends to a virtual bundle $x_{Y'}$ on
    $Y'$, $\Psi_{\mc E}$ is given by the rough excess isomorphism:

    \begin{eqnarray*} \label{excess:forcartierdivisors}
    \Psi_{\mc F_{Y'}, \sigma} : g^! f_!(x) & \stackrel{\sigma}\simeq & \lambda_{-1} (g^* \calo(-Y)) \otimes g^! (
    x_{Y'}) \\
    & \stackrel{g = f{g'}, f' = \op{Id}}\simeq & \lambda_{-1} ({g'}^* {f'}^* \calo(-Y)) \otimes {g'}^! {f'}^! x_{Y'}  \\
    & \simeq & \lambda_{-1} (E) \otimes {g'}^! x.
    \end{eqnarray*}
    \item \textsc{Composition:} Suppose we are given the composition of an upper Cartesian diagram $\mc E$ and a lower Cartesian diagram $\mc
    E'$ (giving $\mc E''$):
    $$\xymatrix{X \ar[r]^g \ar[d]^{e'} & Y \ar[d]^e \\
    X' \ar[r]^{g'} \ar[d]^{f'} & Y' \ar[d]^{f} \\
    X'' \ar[r]^{g''}  & Y''  \\}$$
    with associated excess-bundles $E, E'$ and $E''$, the following diagram commutes
    $$\xymatrix{ (f' \circ e')_! (\lambda_{-1}(E'') \otimes {g}^! x) \ar@/{}^{3pc}/[ddddr]^{\Psi_{\mc E''}} \ar[d]^{f' \circ e' = f'  e'} \\
    f'_! e'_! (\lambda_{-1}(E'') \otimes {g}^! x)  \ar[d]_{ 0 \to E \to E'' \to {e'}^* E' \to 0 } \\
    f'_!( \lambda_{-1}(E') \otimes e'_! (\lambda_{-1}(E) \otimes {g}^! x))  \ar[d]^{\Psi_{\mc E}} \\
    f'_!( \lambda_{-1}(E') \otimes {g'}^! e_! (x) \ar[d]^{\Psi_{\mc E'}}) \\
    {g''}^! f_! e_! (x) \ar[r]^{f \circ e = fe} & {g''}^! (f \circ e)_! (x)}$$
    where we use the projection formula on second left up-to-down arrow.
This will be written symbolically as
$$\Psi_{\mc E'}\Psi_{\mc E} = \Psi_{\mc E''}$$
        \item \label{condition:trivialization} {\textsc{Trivialization:} Suppose $f$ is a \lci projective morphism and let $x$ be a virtual vector bundle which admits a trivialization. Then $f'_!\left( \lambda_{-1}(E) \otimes {g'}^!(x) \right)$ and $g^! f_!(x)$ are both canonically trivialized. We demand that the excess isomorphism interchanges these trivializations.}
\end{enumerate}
\end{T}
\begin{remark} The last property $(e)$ is a consequence of the constraint that the excess isomorphism should be a natural transformation of additive functors (thus in particular respecting trivializations). We state it as a separate property since it is the only condition in the definition of additive functors which is not trivial to verify using Theorem \ref{Thm:deligne-virtual}.
\end{remark}
It moreover follows (see Proposition \ref{Prop:tor-independentexcess}), that if $E$ is trivial, then this isomorphism is necessarily given by Tor-independent base change. We also record the following corollary of the theorem:
\begin{Cor} \label{cor:excessselfintersection}[Self-Intersection Formula, compare proof of \cite{0cyclesarithmeticschemes}, Lemma 3.2] Let $i: Y \to Y'$ be
a regular closed embedding with $Y'$ being an algebraic stack with the resolution property. Then we have a functorial isomorphism
$$i^! i_! (x) \simeq \lambda_{-1}(N_{Y/Y'}^\vee) \otimes x. $$
\end{Cor}
\begin{proof} Indeed, take $X = X' = Y$, $f = g = i$, $f' = g' =
\op{Id}$ and use that the excess-bundle is just the conormal-bundle.
\end{proof}
Below we show the uniqueness of the isomorphism in the case of closed immersions. After this we provide a construction suggested by the proof of unicity to construct the isomorphism for closed immersions. We then prove that the isomorphism is necessarily given by base change in the case of projective bundle projections. One then needs to show that for general projective morphisms, the excess isomorphism obtained by a factorization does not depend on choice of factorization and thereby define the isomorphism in general. \\

\section{The case of closed immersions, uniqueness}

The main object of this section is to prove the following theorem:

\begin{T} If $f$ is a closed immersion, then the excess isomorphism, if it exists, is uniquely determined by the conditions $(a) - (e)$ in the theorem.
\end{T}

\begin{proof} First of all, denote by $\mc E$ the Cartesian diagram $$\xymatrix{ \bb P^1_{X} \ar[r]^{G'} \ar[d]^{F'} & \bb P^1_{Y} \ar[d]^F \\
M' \ar[r]^{G} & M}$$ where $M'$ and $M$ are the deformations to the
normal cone of $f'$ and $f$ respectively. Denote
by $\mc E_0$ and $\mc E_\infty$ the two following diagrams:
\begin{equation} \label{modeldiagrams} \xymatrix{X \ar[r]^{g'} \ar[d]^{f'} & Y \ar[d]^f & X \ar[d]^{\tilde f'} \ar[r]^{g'} &  Y \ar[d]^{\tilde f} \\
X' \ar[r]^g & Y' & \bb P(N' \oplus 1)  \ar[r]^{\tilde g} &
\bb P(N \oplus 1)}\end{equation} where $\tilde{f},\tilde{f}'$ are the canonical zero-sections with projections
$\pi'$ and $\pi$, $\tilde g$ being given by the morphisms induced by an
inclusion of vector bundles $N' \subseteq {g'}^*N.$   In the following we
abuse notation a bit to show unicity, but only in the sense that was used in the formulation of the theorem. Thus, for example, $\Psi_{\mc E}$ denotes an excess isomorphism satisfying the properties of the theorem for the diagram $\mc E$ and for a functor $F$, $F \Psi_{\mc E}$ denotes the image of the excess isomorphism under $F$. We have embeddings $i_0: \mc E_0 \to \mc E$ and $i_\infty: \mc E_\infty \to \mc E$
 and a projection $\Pi: \mc E \to \mc E_0$, satisfying $\Pi \circ i_0 = \op{Id}$. By property (a) of the theorem we see that
 $\Psi_{\mc E_0} i_0^! = i_0^! \Psi_{\mc E}$ and $\Psi_{\mc E_\infty} i_\infty^! = i_\infty^! \Psi_{\mc E}$. Also, applying the natural
transformations $\op{Id} \simeq i_{0}^! \Pi^!$ and $\op{Id} \simeq \Pi_! Ri_{0,!}$ one sees that the functor $\Psi_{\mc E_0}$ is
determined by the functor $i_{0,!} i_0^! \Psi_{\mc E} \Pi^!$. Fixing a
rational function $\lambda$ on $\bb P_{\bb Z}^1$ with divisor $(0) - (\infty)$
defines an isomorphism
\begin{equation}\label{eq:linearequiv} \calo(X') \stackrel{\lambda}\simeq \calo(P(N' \oplus 1) \cup D') \simeq
\calo(\bb P(N' \oplus 1))\otimes \calo(D')
\end{equation}
 of line bundles on $M'$. Here $\bb P(N' \oplus 1)$ and $D'$ are the two components of the blow-up at infinity in the deformation to the normal cone, intersecting in $\bb P(N')$, and the image of $X$ does not intersect that of $D'$ (see Appendix). This rational section is defined up to sign, which will not be seen on the level of the strictly commutative virtual category. We have natural Koszul resolutions
\begin{eqnarray*} 0 \to  \calo(-X') \to  \calo \to  \calo_{X'} \to 0 \\
0 \to  \calo(-D') \to  \calo \to  \calo_{D'} \to 0 \\
0 \to  \calo(-\bb P(N' \oplus 1)) \to  \calo \to  \calo_{\bb P(N' \oplus 1)} \to 0
\end{eqnarray*}
and a resolution
\begin{eqnarray*}
0 \to  \calo(-D')\otimes \calo(-\bb P(N' \oplus 1)) \to
\calo(-D')\oplus \calo(-\bb P(N' \oplus 1)) \to \calo \to \calo_{\bb P(N')} \to 0
\end{eqnarray*}
which together define an isomorphism, via (\ref{eq:linearequiv}) (see also \cite{Riemann-RochAlgebra}, Proposition 4.4)

\begin{eqnarray*} 0 & \simeq & (1-\calo(-X')) - (1-\calo(-D')) - (1-\calo(-\bb P(N' \oplus 1)) \\ & & + (1-\calo(-\bb P(N' \oplus 1))\otimes(1-\calo(-D')) \\
& \simeq & \calo_{X'} - \calo_{D'} - \calo_{\bb P(N' \oplus 1)} +
\calo_{\bb P(N')}
\end{eqnarray*}

and hence an isomorphism
$$ \calo_{X'} \simeq  \calo_{D'} + \calo_{\bb P(N' \oplus 1)} - \calo_{\bb P(N')}.$$
Via the projection formula this gives an equality
\begin{eqnarray*} i_{0, !} \Psi_{\mc E_0} & \simeq & i_{0, !} i_0^!  \Psi_\mc E \Pi^! \simeq \calo_{X'} \otimes  \Psi_\mc E \Pi^! \\
& \simeq & \calo_{\bb P(N' \oplus 1)} \otimes \Psi_\mc E \Pi^!  + \calo_{D'} \otimes \Psi_\mc E \Pi^! -  \calo_{\bb P(N')} \otimes \Psi_\mc E \Pi^!
\end{eqnarray*}

Since $D'$ does not intersect the image of $\Pi^*$, the conditions (b) and (e) of the theorem determine $\calo_{D'} \otimes \Psi_\mc E \Pi^!$ and $\calo_{\bb P(N')} \otimes \Psi_\mc E \Pi^!$ and thus $i_{0, !} \Psi_{\mc E_0}$ is determined by $$\calo_{\bb P(N' \oplus 1)} \otimes \Psi_\mc E \Pi^! = i_{\infty, !} i_\infty^! \Psi_\mc E \Pi^!=  i_{\infty, !} \Psi_{\mc E_\infty} i_\infty^! \Pi^! .$$
We recall from Lemma 4.5, chapter V, \cite{Riemann-RochAlgebra} (see the proof for this slightly more refined statement):

\begin{TL} \label{lemma:transversality} Let $F: P \to M$ be a regular embedding, and let $\Phi: Y \to M$ be a representable morphism, and consider the fiber square
$$\xymatrix{ X \ar[r]^f \ar[d]^\phi & Y \ar[d]^\Phi \\
P \ar[r]^F & M}.$$ If $f$ is a regular embedding of the same codimension as $F$, then this square is Tor-independent.
\end{TL}
 Note also that by symmetry, the same conclusion holds
with $\Psi$ and $\phi$ in place of $F$ and $f$.
Now, the diagrams that arise for application of condition (a) are the following:
$$\xymatrix {X \ar[r] \ar[d] & \bb P^1_X \ar[d] & X \ar[d] \ar[r] &  \bb P^1_X \ar[d] &  Y \ar[r] \ar[d] & \bb P^1_Y \ar[d] & Y \ar[r] \ar[d] & \bb P^1_Y \ar[d] \\
X' \ar[r] & M' & \bb P(N' \oplus 1)  \ar[r] &  M' & Y' \ar[r] & M
& \bb P(N \oplus 1)  \ar[r] &  M} $$ The morphisms $\bb P^1_X
\hookrightarrow M', \bb P^1_Y \hookrightarrow M$ are both regular
by Theorem 4.5, chapter IV, \cite{Riemann-RochAlgebra}, all
diagrams Cartesian and codimension is preserved so we can apply
the above lemma. Hence all diagrams are Tor-independent and we can
apply (a).
\\ Thus we are reduced to showing uniqueness for the diagram $\mc
E_\infty$. We proceed by induction on the dimension of $E$. In
case the dimension of $E$ is 0, then indeed the diagram is
Tor-independent and the isomorphism is fixed by Tor-independent
base change, which is a consequence of condition (a). In the case $\rk E > 0$, consider the flag variety $p: G = \op{Gr}_{1,m, N} \to Y$ parameterizing
flags $L \subseteq M \subseteq N$, with $L, M$ of rank $1,m$ and
$m = \op{rk} N'$. Then $p^!: V(Y) \to V(G)$ is faithful (this follows from injectivity on the level of $K_1$, which follows as in the classical case from the projective bundle case, cf. \cite{SGA6}, VI, section 4), and an easy verification shows that by the Tor-independent base change property we can assume our diagram is of the form $$\xymatrix{\op{Gr}_{1,m,{g'}^* N} \ar[r]^{h'} \ar[d] & \op{Gr}_{1,m, N} \ar[d] \\
\bb P({h'}^* \mc M) \ar[r] & \bb P(p^* N)}$$ where $\mc L \subset
\mc M \subset p^* N$ is the universal flag on $G$. Suppose we can
filter the vector bundle $p^*N/\mc M$ by a maximal flag on $G$ which is then in particular a
flag of $p^* N$ including $M$. Then we can compose our big diagram as a composition of smaller diagrams which
are either Tor-independent or codimension 1-cases like in (c). Here we notice that $\bb P(p^* N)$ is a projective bundle over $\op{Gr}_{1,m,N}$ and so any virtual bundle actually extends as in (c). By (d)
this does not depend on the choice of flag and the assertion of the splitting principle (cf. Section \ref{DRR1-section:splittingprinciple} in \cite{DRR1}) is that this isomorphism descends from a maximal flag variety and we conclude . \\\\
\section{The case of closed immersions, rougher excess and existence} The previous
section gave a recipe for the construction of the \linebreak excess isomorphism, which we spell out. Let $M$ (resp. $M'$) be the deformation to the normal cone of $f$ (resp. $f'$). Denote by $i_D, i_{D \cap \bb P(N \oplus 1)}$ (resp. $i_{D'}, i_{D' \cap \bb P(N' \oplus 1)}$) the
closed immersions of $D$ and $D \cap \bb P(N \oplus 1)$ (resp. $D'$ and $D' \cap \bb P(N' \oplus 1)$) in $M$ (resp. $M'$). Also denote by $\pi, \pi', \widetilde {\pi}$ and $\widetilde{\pi'}$ the natural projective bundle projections of the various $\bb P(? \oplus 1)$, and consider the diagram
given by a cubical diagram as in (a). By the universal property of blow-ups (\cite{intersection}, Appendix B.6.9) we obtain commutative diagram
$$   \xymatrix{
 & \bb P^1_{\widetilde X} \ar@{->}[rr]^{\widetilde G'} \ar@{->}'[d][dd]^{Q''}
   & & \bb P^1_{\widetilde Y} \ar@{->}[dd]^{Q}
\\
 \widetilde{M}'
 \ar@{<-}[ur]^{\widetilde F'} \ar@{->}[rr]^(0.40){\widetilde G'}\ar@{->}[dd]^{Q'''}
 & & \widetilde{M} \ar@{<-}[ur]^{\widetilde F}\ar@{->}[dd]^(0.35){Q'}
\\
 & \bb P^1_X \ar@{->}'[r][rr]^{G }
   & & \bb P^1_Y
\\
 M' \ar@{->}[rr]^{ G'}\ar@{<-}[ur]^{F' }
 & & M \ar@{<-}[ur]^{F}
}$$
where all the squares are Cartesian, except possibly the
front and back vertical ones. We construct a sort of preliminary excess isomorphism
for the diagram $\mc E_\infty$ considered in (\ref{modeldiagrams}). If $\xi$ (resp. $\xi'$) denote $\pi^* N(1)$ (resp. ${\pi '}^*  {N'}(1)$) on $\bb P(N \oplus 1)$
(resp. $\bb P(N' \oplus 1)$), we have a short exact sequence
\begin{equation}\label{sequence:quotientbundle}0 \to \bb E \to {{\tilde g}^*} \xi^\vee \to (\xi')^\vee \to 0
\end{equation}
 where $\bb E$ is an extension of the excess bundle of the diagram, more precisely ${\pi'}^* E(-1)$. Then we have a canonical Koszul resolution, obtained from the regular section of $\xi$ determined by $\calo(-1) \to \pi^* (N \oplus 1) \to p^* N$,
\begin{equation}\label{sequence:Koszulres}0 \to \wedge^n \xi^\vee \to \ldots \to \xi^\vee \to \calo_{\bb P(N \oplus 1)} \to \tilde f_* \calo_Y \to 0.
\end{equation}
For the construction, we work with a vector bundle $V$ first be a vector bundle on $X$. Since tensoring with
a vector bundle preserves exactness, there is a Koszul resolution:
\begin{equation}\label{sequence:Koszulresgens}\wedge^\bullet \xi^\vee \otimes \pi^* x \to
\tilde f_* \calo_Y \otimes \pi^* x = \tilde f_* (V).
\end{equation}


Applying $\tilde{g}^!$ to this we obtain \begin{eqnarray} \label{isomodelsituation} \tilde{g}^! \tilde f_! V & = & \tilde{g}^! \left( \sum (-1)^i \wedge^i \xi^\vee \otimes \pi^! V \right) \\
& = & \left(\sum (-1)^i \wedge^i \tilde{g}^* \xi^\vee \otimes \tilde{g}^! \pi^! V \right) \nonumber \\
& = & \lambda_{-1} (\tilde{g}^* \xi^\vee) \otimes {\pi'}^! g^! V \nonumber \\
& = & \lambda_{-1} ((\xi')^\vee) \otimes \lambda_{-1}(\bb E) \otimes {\pi'}^!   g^! V \nonumber \\
& = & \tilde f'_!\left(\lambda_{-1}(E) \otimes g^! V \right). \nonumber
\end{eqnarray}
The isomorphism from the third and fourth line comes from
(\ref{sequence:quotientbundle}), the isomorphism between the fourth
and the fifth come from a Koszul-resolution and the projection formula similar to that of
(\ref{sequence:Koszulres}). This is the rough excess isomorphism already exhibited in Section \ref{section:roughexcess}. \\
In the case of a general diagram $\mc E$ $$ \xymatrix{X \ar[r]^{g'} \ar[d]^{f'} & Y \ar[d]^{f} \\
X' \ar[r]^{g} & Y' }$$
the isomorphism $\Psi_\mc E$ is defined in the following fashion, using the notation
of the previous section:
\begin{eqnarray*}  i_{0, !} g^! f_! (V) \\  \simeq & i_{0, !} g^! f_! i_0^! \Pi^! (V)  \\
 \simeq & i_{0, !} g^! i_0^! F_! \Pi^!(V) \\
 \simeq & i_{0, !} i_0^! G^! F_! \Pi^!(V) \\
 \stackrel{\lambda}\simeq & i_{\infty, !} i_\infty^! G^! F_! \Pi^!(V) + \\ & i_{D', !} i_{D'}^!G^! F_! \Pi^!(V) - \\
& i_{D' \cap P(N' \oplus 1), !}i_{D' \cap \bb P(N' \oplus 1)}^! G^! F_! \Pi^!(V) \\
 \simeq & i_{\infty, !} i_\infty^! G^! F_! \Pi^!(V) \\
 \simeq & i_{\infty, !} \tilde g^! \tilde f_!(V) \\
 \simeq & i_{\infty, !} \tilde f'_!(\lambda_{-1}(E) \otimes g^!(V))
\end{eqnarray*}
The term $i_{D', !} i_{D'}^! G^! F_! \Pi^!(V)$ is isomorphic by Tor-independent base change to \linebreak $i_{D', !} i_{D' }^! {\Pi}^! g^! f_!(V)$ and is
canonically trivialized since the intersection of $D'$ and $D' \cap \bb P(N' \oplus 1)$ with $Y$ is empty. The second isomorphism is the rough excess isomorphism
of (\ref{isomodelsituation}).
Applying $\Pi_!$ on both sides gives the required isomorphism, since
$\Pi i_{\infty} \tilde f' = f'$ and $(\Pi i_{0})_{X'}
= \op{Id}_{X'}$. We claim this "rougher excess isomorphism" satisfies the properties of the theorem. The compatibility with the projection formula in condition (b) is established in a similar way and essentially follows from the naturality of the construction and compatibility with base change, and the interested reader will have no problem verifying it with what follows below. We need to verify the other properties.
\begin{Prop} Let $$   \xymatrix@!{
 & \widetilde X \ar@{->}[rr]^{\widetilde {g'}} \ar@{->}'[d][dd]^{q''}
   & & \widetilde Y \ar@{->}[dd]^{q}
\\
 \widetilde{X'} \ar@{<-}[ur]^{\widetilde f'} \ar@{->}[rr]^(0.40){\widetilde g}\ar@{->}[dd]^{q'''}
 & & \widetilde{Y'} \ar@{<-}[ur]^{\widetilde f}\ar@{->}[dd]^(0.35){q'}
\\
 & X \ar@{->}'[r][rr]^{g' }
   & & Y
\\
 X' \ar@{->}[rr]^{ g}\ar@{<-}[ur]^{f' }
 & & Y' \ar@{<-}[ur]^{f}
}$$
be a commutative cube as in condition (a) of Theorem \ref{thm:excess}. Then the rougher excess isomorphism satisfies the conclusion of \ibid.
\end{Prop}

\begin{proof} This will essentially be by functoriality of the blow-up construction. Keep the notation as introduced. In the cube in the introduction of this section the fiber over $0$ is the diagram we start with, whereas the fiber at $\infty$ is,

$$   \xymatrix{
 & \widetilde X \ar@{->}[rr]^{\tilde g} \ar@{->}'[d][dd]^{q''}
   & & \widetilde Y \ar@{->}[dd]^{q}
\\
 \bb P(\widetilde {N}' \oplus 1)
 \ar@{<-}[ur]^{\tilde f'} \ar@{->}[rr]^(0.40){\tilde g'}\ar@{->}[dd]^{\widetilde q'''}
 & & \bb P(\widetilde {N} \oplus 1) \ar@{<-}[ur]^{\tilde f}\ar@{->}[dd]^(0.35){\widetilde q'}
\\
 & X \ar@{->}'[r][rr]^{g }
   & & Y
\\
 \bb P(N' \oplus 1) \ar@{->}[rr]^{ \widetilde g'}\ar@{<-}[ur]^{\widetilde f' }
 & & \bb P(N \oplus 1) \ar@{<-}[ur]^{\widetilde f}
}$$
minus some unwanted factors. The full fiber over $\infty$ would include the factors $D, D', \widetilde D, \widetilde D'$, which don't meet the images of $Y, X,
\widetilde Y$ and $\widetilde X$ respectively. Let's denote by $i_0, i_0', i_0''', \tilde i_0, \tilde i_0', \tilde i_0'''$ the inclusions
of $Y, Y', X', \widetilde Y, \widetilde Y'$ and $\widetilde X'$ in
$\widetilde M'$ and $\widetilde M$ respectively. Also denote by
$i_\infty, i_\infty', i_\infty''', \widetilde i_\infty, \widetilde
i_\infty', \widetilde i_\infty'''$ (resp. $i_{\widetilde D}$ and
$i_{\widetilde D'}$) the inclusions over $\infty$ of $Y, \bb P({N}
\oplus 1), \bb P({N'} \oplus 1), \widetilde Y, \bb P(\widetilde {N}
\oplus 1), \bb P(\widetilde {N'} \oplus 1)$ (resp. $\widetilde D'$
and $\widetilde D$) in $\bb P^1_Y, M, M', \bb P^1_{\widetilde Y},
\widetilde M, \widetilde M'$ (resp. $\widetilde M$ and $\widetilde
M'$). Also introduce, for the purpose of this section, the natural projections $\Pi: \bb P^1_Y \to Y$ and
$\overline \Pi: \bb P^1_{\widetilde Y} \to \widetilde Y$.

Then we find that the following diagram is commutative:
$$\xymatrix @C=6pc {\overline g^! \overline f_! q^! \ar[r]^{base~change} \ar[d]^{\overline \Pi
\overline i_0 = \op{Id}} & \overline g^! {q'}^! f_! \ar[r]^{q' \tilde g = g q'''}  &  {q'''}^! g^! f_!  \ar[d]^{\Pi i_0 = \op{Id}} \\
\ar[d]^{base~change} \overline g^! \overline f_!  \overline i_0^! \overline \Pi^! q^!  &  &  {q'''}^! g^! f_! i_0^! \Pi^! \ar[d]^{base~change} \\
\ar[d]^{\overline G \overline{i_0'''} = \overline{i_0'} g }
\overline g^! \overline {i_0'}^! \overline F_! \overline \Pi^! q^!
\ar[r]^{\Pi Q = q \tilde \Pi} & \overline g^! \overline {i_0'}^!
\overline F_! {Q}^! \Pi^! \ar[d]_{base~change} & {q'''}^! i^! {i_0'}^! F_! \Pi^! \ar[dd]^{G i_0'''= i_0' g} \\
 \overline {i_0'''}^! \overline G^! \overline  F_! \overline \Pi^! q^! \ar[d]^{\Pi Q = q \overline \Pi} & \overline g^! \overline {i_0'}^! {Q'}^! F_! \Pi^! \ar[ur]_{Q' i_0' \overline g = i_0'  g q'''}\\
\ar[r]^{base~change} \overline {i_0'''}^! \overline G^! \overline
F_! Q^! \Pi^! & \ar[r]^{Q' \overline G \overline {i_0'''} = G i_0'''
q'''} \overline {i_0'''}^! \overline G^! {Q'}^! F_! \Pi^! & {q'''}^!
{i_0'''}^! G^!  F_! \Pi^! }$$ This follows from Lemma \ref{lemma:basechange}, which asserts that composition of base changes is base change itself.

We apply $\overline i_{0, *}'''$ to the above diagram and continue: \\
$$\xymatrix @C=9pc {\overline i_{0, *}''' \overline {i_0'''}^! \overline G^! \overline F_! \overline \Pi^! q^!  \ar[d]^{projection~formula} \ar[r]^{\hbox{see above}} &
\overline i_{0, !}''' {q'''}^! i_{0, !}'''  G^! F_! \Pi^! \ar[d]^{{base~change}} \\
\calo_{\overline X'} \otimes \overline G^! \overline F_! \overline \Pi^! q^! \ar[d]^{\lambda} & {Q'''}^!  i_{0, !}''' {i_{0}'''}^!  G^! F_! \Pi^! \ar[d]^{\lambda + {projection~formula}} \\
\left (\overline{K} \otimes
\overline G^! \overline F_! \overline \Pi^! q^! \right) \ar[r] \ar[d] & {Q'''}^!\left( K  \otimes G^! F_! \Pi^! \right) \ar[d] \\
\calo_{\bb P(\overline N' \oplus 1))} \otimes \overline G^! \overline F_!
\overline \Pi^! q^! \ar[d]^{projection~formula} \ar[r] &
{Q'''}^! \left (\calo_{\bb P( N' \oplus 1))} \otimes G^! F_! \Pi^! \right )
\ar[d]^{projection~formula}
 \\
\overline {i_\infty'''}_! \overline {i_\infty'''}^! \overline G^! \overline F_! \overline \Pi^! q^! \ar[r] \ar[d] & {Q'''}^! {i_\infty'''}_! {i_\infty'''}^! G^! F_! \Pi^! \ar[d] \\
\overline {i_\infty'''}_! g^! f_! q^! \ar[r] & \overline
{i_\infty'''}_! {q'''}^! g^! f_! .} $$
Here $\overline K = \calo_{\bb P(\overline N' \oplus 1)} + \calo_{\overline D} -
\calo_{\bb P(\overline N' \oplus 1) \cap \overline D}$ and $K = \calo_{\bb
P( N' \oplus 1)} + \calo_{ D} - \calo_{\bb P( N' \oplus 1) \cap D}$. The upper square in this diagram is commutative by our choice of $\lambda \in k(\bb P^1_{\bb Z})$ and by verifying (cf. Lemma \ref{lemma:projformula}) that the base changes involved commute with the projection formula. In the next square, the induced isomorphism induces a commutative diagram of isomorphisms:

$$\xymatrix{\calo_{\overline X'} \ar[r] \ar[d]^\lambda & {Q'''}^! \calo_{X'} \ar[d]^\lambda\\
\calo_{\bb P(\overline N'+1)} + \calo_{\overline D} - \calo_{\bb
P(\overline N'+1)) \cap \overline D} \ar[r] & {Q'''}^!(\calo_{\bb P(
N'+1)} + \calo_{ D} - \calo_{\bb P( N'+1) \cap D}).}$$ Moreover,
in the second line there is an induced isomorphism

$$\calo_{\bb P(\overline N'+1)} \simeq {Q'''}^! \left (\calo_{\bb
P( N'+1)} \right )$$ and thus also an induced isomorphism of
superfluous terms (see Lemma \ref{lemma:coneterms}):
$$\calo_{\overline D} - \calo_{\bb
P(\overline N'+1) \cap \overline D} \simeq {Q'''}^!\left (\calo_{
D} - \calo_{\bb P( N'+1) \cap D} \right ).$$ We conclude that the next square is also commutative. \\ Finally, the lower square is commutative for the same reason as the other first square above, i.e. by Lemma \ref{lemma:basechange} cited. Applying $\overline \Pi_!$ as in the definition of our morphism we obtain the commutativity, modulo showing commutativity in the model situation. \\
We are left to consider the cube
$$   \xymatrix{
 & \widetilde X \ar@{->}[rr]^{\widetilde {g'}} \ar@{->}'[d][dd]^{q''}
   & & {\widetilde Y} \ar@{->}[dd]^{q}
\\
 \bb P(\widetilde{N}' \oplus 1)
 \ar@{<-}[ur]^{\widetilde f'} \ar@{->}[rr]^(0.40){\widetilde g}\ar@{->}[dd]^{q'''}
 & &  \bb P(\widetilde{N} \oplus 1) \ar@{<-}[ur]^{\widetilde f}\ar@{->}[dd]^(0.35){q'}
\\
 &  X \ar@{->}'[r][rr]^{g'}
   & &  Y
\\
 \bb P(N' \oplus 1) \ar@{->}[rr]^{g}\ar@{<-}[ur]^{f' }
 & & \bb P(N \oplus 1)  \ar@{<-}[ur]^{f}
}$$ with the notation of the left diagram of (\ref{modeldiagrams}). Denote by $$\xi, \xi', \widetilde \xi, \widetilde \xi'$$ the
analogues of the bundles already introduced on $$\bb P(N \oplus 1), \bb P(N' \oplus
1), \bb P(\widetilde N \oplus 1), \bb P(\widetilde N' \oplus
1)$$ respectively, and let $\bb E$ and $\widetilde {\bb E}$ denote the natural extensions of the excess bundles. The proposition is that
$$\xymatrix{\tilde{f}'_!(\lambda_{-1}(\tilde E) \otimes \tilde g'^!(q^! x)) \ar[d] \ar[r] & \tilde g^! \tilde f_!(q^!(x)) \ar[d] \\
\tilde{f}'_!(\lambda_{-1}({q''}^* E) \otimes ({q''}^! {g'}^! x)) \ar[d] & \tilde g^! {q'}^! f_!(x) \ar[dd] \\
\tilde{f}'_! {q''}^! (\lambda_{-1}(E) \ar[d] \otimes {g'}^! x)) &  \\ 
{q'''}^!(f_!(\lambda_{-1}(E)  \otimes {g'}^! x)) \ar[r] &  {q'''}^!
g^! f_!(x)}$$ is a commutative diagram, where the horizontal morphisms are the already constructed candidates for the excess isomorphisms in (\ref{isomodelsituation}). We show it is commutative by breaking it up into smaller pieces. Consider:
$$\xymatrix  {\widetilde g^! \widetilde f_! q^! \ar[r]^{base~change} \ar[d]^{resolution} & \widetilde g^! {q'}^! f_! \ar[r]^{q' \widetilde g = g q'''} &  {q'''}^! g^! f_! \ar[d]^{resolution}\\
\widetilde g^! (\lambda_{-1} (\widetilde \xi^\vee) \otimes \widetilde
\pi^! q^!) \ar[r] \ar[d] &
\ar@{=}[ld] \lambda_{-1} (\widetilde g^! \widetilde \xi^\vee) \otimes \widetilde g^! \widetilde \pi^! q^! \ar[r] & {q'''}^! g^! (\lambda_{-1}(\xi^\vee) \otimes \pi^!) \ar[d] \\
\lambda_{-1}(\widetilde g^* \widetilde \xi^\vee) \otimes \widetilde g^! \widetilde \pi^! q^! \ar[rr] &
&  \lambda_{-1}({q'''}^* g^* \xi^\vee) \otimes {q'''}^! g^! \pi^! }$$
and
$$
\xymatrix {\ar[d]^{0 \to \widetilde {\bb E} \to \widetilde {g}^* \tilde {\xi}^\vee \to \tilde{\xi'}^\vee \to 0}
\lambda_{-1}(\widetilde g^* \widetilde \xi^\vee) \otimes \widetilde g^! \widetilde \pi^! q^! \ar[r]
& \ar[d]^{0 \to \bb E \to g^* \xi^\vee \to  {\xi'}^\vee \to 0} \lambda_{-1}({q'''}^* g^* \xi^\vee) \otimes {q'''}^! g^! \pi^! \\
\ar[d]^{projection~formula} \lambda_{-1}(\widetilde {\xi'}^\vee) \otimes \lambda_{-1}(\widetilde{\bb E})
\otimes \widetilde g^! \widetilde \pi^! q^! \ar[r] & \ar[d]^{projection~formula} {q'''}^! \left ( \lambda_{-1}({\xi'}^\vee) \otimes \lambda_{-1}(\bb E)  \otimes g^! \pi^! \right ) \\
 \widetilde f_! \left (\lambda_{-1}(\widetilde E) \otimes \widetilde {g'}^! q^! \right ) \ar[d]^{{q''}^* E = \widetilde E} &   {q'''}^! \left (f_!(\lambda_{-1}(E) \otimes {g'}^!) \right) \\
 \widetilde f_! \left (\lambda_{-1}({q''}^* E) \otimes {q''}^! {g'}^! \right )
 \ar[ur]^{base~change}}$$
 defined in the way indicated. The first lower diagram commutes by general nonsense; the isomorphisms are just given by certain natural transformations. To see why the second upper diagram commutes, consider the diagram
$$\xymatrix{\label{sequence:quotientbundles}0 \ar[r] & {q'''}^* \bb E \ar[r] \ar[d] & {q'''}^* g^* \xi^\vee \ar[r] \ar[d] & {q'''}^* (\xi')^\vee \ar[r] \ar[d] &  0 \\
0 \ar[r] & \widetilde {\bb E} \ar[r] &  \widetilde g^* \widetilde \xi^\vee \ar[r] &  (\widetilde \xi')^\vee \ar[r] & 0} $$
where all the vertical maps are isomorphisms. Indeed, all the maps exist by functoriality, and they are isomorphisms by the assumption on Tor-independence. The remaining diagrams are commutative again by applying Lemma \ref{lemma:basechange}.
\end{proof}

\begin{Prop}\label{Prop:composition} Let $e: Y \to Y', f: Y' \to Y''$ be two regular closed immersions, and $g: X'' \to Y''$ an arbitrary representable morphism, with associated diagrams
$\mc E, \mc E'$ and big diagram $\mc E''$,

$$\xymatrix{X \ar[r]^{g''} \ar[d]^{e'} \ar @{} [dr] |{\mc E}
& Y \ar[d]^e \\
X' \ar[d]^{f'} \ar[r]^{g'} \ar @{} [dr] |{\mc E'}
 & Y' \ar[d]^f \\
X'' \ar[r]^g & Y''.}$$ We also suppose $e'$ and $f'$ are regular closed immersions. Then $\Psi_{\mc E''}$ is the
composition of $\Psi_{\mc E'}$ and $\Psi_{\mc E}$ in the sense of Theorem \ref{thm:excess}, condition (d).
\end{Prop}

\begin{proof} By a deformation to the normal cone argument, as above, we can
suppose that our immersions are of the form

$$\xymatrix{ X \ar[r]^{g''} \ar[d]^{e'} & Y \ar[d]^e \\
\bb P_X(N' \oplus 1) \ar[r]^{g'} \ar[d]^{f'} & \bb P_Y(N \oplus 1) \ar[d]^f \\
\bb P_X(M' \oplus 1) \ar[r]^g & \bb P_Y(M \oplus 1).}$$
Denote by $p: \bb P_Y(M \oplus 1) \to Y, p': \bb P_{X}(M' \oplus 1) \to X, \pi:\bb P_Y(N \oplus 1) \to Y$ and $\pi': \bb P_{X}(N' \oplus 1) \to X$
the natural projections. Define $$N^\bot := \ker[M^\vee \to N^\vee] = (M/N)^\vee, {N'}^\bot := \ker[{M'}^\vee \to {N'}^\vee] = (M'/N')^\vee.$$
As before, the morphism $\calo(-1) \to p^* M \oplus 1 \to p^* (M)$ defines a regular
section of $p^* M(1)$ which vanishes exactly at $Y$, and in the same
way we get a regular section of the vector bundle $p^* (M/N)(1)$
which vanishes exactly at $\bb P_Y(N \oplus 1)$ (see
\cite{intersection}, Appendix B. 5.6). We have the following
commutative diagram with exact columns and lines
\begin{align}
\xymatrix{ & 0 \ar[d] & 0 \ar[d] & 0 \ar[d] & \\
0 \ar[r] & \bb E' \ar[r] \ar[d] & {g}^* p^* N^\bot(-1) \ar[r] \ar[d] & {p'}^* {N'}^\bot(-1) \ar[d] \ar[r] &  0 \\
0 \ar[r] & \bb E'' \ar[r] \ar[d] & {g}^* p^* M^\vee(-1) \ar[r] \ar[d] & {p'}^* {M'}^\vee(-1)\ar[r] \ar[d] & 0 \\
0 \ar[r] & \bb E \ar[r] \ar[d] & {g}^* p^* N^\vee(-1) \ar[r] \ar[d] &
{p'}^*
{N'}^\vee(-1) \ar[r] \ar[d]& 0 \\
& 0 & 0 & 0 } \label{xymatrix:bigsquare}
\end{align} Here $\bb E, \bb E', \bb E''$ denote vector bundles
on $\bb P(M' \oplus 1)$ which extend the excess-bundles $E, E'$ and $E''$
respectively, as before. Using the canonical Koszul resolutions, the proposition is that the following diagram is commutative:

\begin{flushleft}
$$\def\objectstyle{\scriptstyle}
\def\labelstyle{\scriptstyle} \vcenter{\scalebox{0.9}[1]{\xymatrix @-0.5pc @C=0.5pc @R=1pc{(f' \circ e')_!
(\lambda_{-1} (E'') \otimes {g''}^! (V)) \ar[r] \ar[d] \ar @{} [ddr] |{\mathbf{(A)}}   &
\lambda_{-1}( \bb E'') \otimes \lambda_{-1}
({p'}^! {M'}^\vee(-1)) \otimes (g'' \circ p')^! (V)) \ar[d] \\
f'_! (e'_! (\lambda_{-1}( {e'}^! E') \otimes \lambda_{-1} ( E) \otimes
{g''}^! (V))) \ar[d] & \lambda_{-1}( \bb E') \otimes \lambda_{-1}( \bb E) \otimes \lambda_{-1}
({p'}^! {M'}^\vee(-1)) \otimes (g'' \circ p')^! (V)) \ar[d] \\
f'_! (\lambda_{-1}(E') \otimes e'_!(\lambda_{-1}(E) \otimes {g''}^! (V)) ) \ar[d] \ar[r] \ar @{} [dr] |{\mathbf{(B)}} &
\lambda_{-1}(\bb E') \otimes \lambda_{-1}({p'}^* {N'}^\bot(-1)) \otimes \lambda_{-1}({p'}^* {N'}^\vee(-1))\otimes \lambda_{-1}(\bb E) \otimes ({p'}^! {g''}^! V) \ar[ddd]   \\
f'_! \left (\lambda_{-1}(E') \otimes \lambda_{-1}({{\pi'}^*
N'}^\vee(-1)) \otimes \lambda_{-1} ({\pi'}^*E) \otimes {\pi'}^! {g''}^! (V) \right) \ar[d] \ar@/{}_{1pc}/[ur] \ar @{} [dr] |{\mathbf{(C)}}  & \\
f'_! \left (\lambda_{-1}(E') \otimes  \lambda_{-1}({g'}^* \pi^*
N^\vee(-1)) \otimes {g'}^! \pi^! (V) \right ) \ar[d] \ar@/{}^{1pc}/[dr] \ar @{} [dr] |{\mathbf{(D)}} &
 \\
f'_! \left (\lambda_{-1}(E') \otimes  {g'}^! e_! (V) \right ) \ar[d] \ar[r] \ar @{} [dr]|-{\mathbf{(E)}} & \lambda_{-1}({p'}^* {N'}^\bot(-1)) \otimes \lambda_{-1}(\bb E') \otimes \lambda_{-1}(g^* {p}^* {N}^\vee(-1)) \otimes ({p'}^! {g''}^! V) \ar[d]
 \\
\left (\lambda_{-1}({p'}^* {N'}^\bot(-1)) \otimes \lambda_{-1}(\mathbb{E}') \otimes \lambda_{-1}(g^* p^* N^\vee(-1)) \otimes g^! p^! V \right ) \ar[d]  \ar@/{}_{0.5pc}/[ur] & \left (\lambda_{-1}(g^* p^* N^\bot(-1)) \otimes \lambda_{-1}(g^! p^! N^\vee(-1)) \otimes g^! p^! V \right )  \ar@/{}_{0.5pc}/[dl] \ar[d] \ar @{} [dl] |{\mathbf{(G)}} \\
 \left (\lambda_{-1}(g^! p^! M^\vee(-1)) \otimes   g^! p^! (V) \right ) \ar[r] \ar @{} [uur] |{\mathbf{(F)}} & g^! (f \circ e)_! (V)} }.} $$
\end{flushleft}
A few comments are in order. Since there is a Koszul resolution determined by $N^\vee(-1) \to \calo_{\bb P_Y(N \oplus 1)}$ of
$Y$ on $\bb P_Y(N \oplus 1)$ the virtual bundle $e_!(V)$ has an extension to $\bb P_Y(M \oplus 1)$ given by
$\lambda_{-1}(p^* N^\vee(-1)) \otimes p^! V$ and the left arrow of diagram $\mathbf{(E)}$ is defined using this. The composition of the left arrows of
the diagrams $\mathbf{(B)}, \mathbf{(C)}$ and $\mathbf{(D)}$ constitute the isomorphism determined by $\mc E$ and finally the composition of the leftmost arrows of
$\mathbf{(E), (F)}$ and $\mathbf{(G)}$ give the isomorphism determined by $\mc E'$. \\
Moreover $\mathbf{(F)}$ commutes since the isomorphism (\ref{lambda1}) is compatible with filtrations and the diagram (\ref{xymatrix:bigsquare}) and the composition of all the rightmost downwards arrows is the isomorphism determined by $\mc E''$ for the same reason. The diagrams $\mathbf{(C)}$ and $\mathbf{(G)}$ commute because of naturality of the Koszul resolution. We refrain from giving all the details of the fact that the other diagrams commute. The interested reader can however easily verify that they do using that all the Koszul resolutions involved are compatible in an obvious sense using the below lemma (a version of this appears already in the Chow-categorical context in \cite{Franke-unpublished}, section 1.3, and the below is inspired by this), which after an inspection takes care of $\mathbf{(A)}$ and  $\mathbf{(B)}$. The other diagrams commute for similar, albeit slightly more involved, reasons. For this, suppose we have vector bundles $V' \subseteq V \subseteq W$ ($V'$ possibly empty) with quotient vector bundles and consider the sequence of inclusions $$\bb P(V' \oplus 1) \to \bb P(V \oplus 1) \to \bb P(W \oplus 1).$$ Denote by $p_{V'}, p_V$ and $p_W$ the canonical projections. On $\bb P(W \oplus 1)$, we have a regular section of $p_W^* (W/{V'})(1)$ (resp. $p_W^* (W/V)(1)$) vanishing on $\bb P(V' \oplus 1)$ (resp. $\bb P(V \oplus 1)$ (cf. \cite{intersection}, B.5.6) related via the natural surjection \linebreak $p_W^* (W/{V'})(1) \to p_W^* (W/V)(1)$. On the kernel of this surjection, $p_W^* (V/{V'})(1)$, restricted to $\bb P(V \oplus 1)$ we have mutatis mutandis a regular section vanishing on $\bb P(V' \oplus 1)$. Notice that per definition we have an exact sequence
\begin{equation} \label{Koszulcompata} 0 \to (W/{V})^\vee \to( W/{V'})^\vee \to (V/{V'})^\vee \to 0.
\end{equation}
\begin{TL}
In the above situation we have a commutative diagram
$$\xymatrix{\bb P(V' \oplus 1) \ar[r]^i \ar[d]^{p_{V'}} & \bb P(V \oplus 1) \ar[r]^j \ar[dl]^{p_V} & \bb P(W \oplus 1) \ar@/{}^{1pc}/[dll]^{p_W} \\
X }$$
where $i$ and $j$ are the natural inclusions. Let $x$ be a virtual vector bundle on \linebreak $\bb P(W \oplus 1).$ Then the claim is that the diagram

$$\xymatrix { \lambda_{-1}(p_W^* (V/{V'})^\vee(-1)) \otimes \lambda_{-1}(p_W^* (W/V)^\vee(-1)) \otimes x \ar[d]^{\Phi} \ar[r]^-{(\ref{Koszulcompata})} & \lambda_{-1}(p_W^*(W/{V'})^\vee(-1)) \otimes x \ar[d]^{\Phi} \\
i_!(\lambda_{-1}(p_V^* (V/V')^\vee(-1)) \otimes i^! x  )
\ar[r]^{\Phi} &  (ji)_!((ji)^! x)}$$
commutes. Here $\Phi$ denotes the respective isomorphisms determined by the regular sections described above.
\end{TL}
\begin{proof} We can assume the virtual bundles are trivial, the isomorphism is obtained by tensoring with $x$ and applying the projection formula. Suppose first that $W/V$ is a line bundle. Then we have exact sequences, for $i > 0$, \begin{equation} \label{eq:exactlambda} 0 \to \Lambda^i (V/V')^\vee(-i) \to \Lambda^i (W/V')^\vee(-i) \to \Lambda^{i-1} (V/V')^\vee(-i+1) \otimes (W/V)^\vee(-1) \to 0, \end{equation}
 compatible with the various Koszul resolutions. The claim follows from a consideration of these resolutions together with the fact that the isomorphism $$\lambda_{-1}((W/V)^\vee(-1)) \otimes \lambda_{-1}((W/V')^\vee(-1)) \simeq \lambda_{-1}((V'/V)^\vee(-1))$$ from Example \ref{example:lambda-1} is induced by (\ref{eq:exactlambda}). The same type of argument applies if $V/V'$ is a line bundle. Suppose now that we are given a vector bundle $V'' \subseteq V'$ with vector bundle quotient $V'/V''$. Then we have a sequence of inclusions $$\bb P(V'' \oplus 1) \hookrightarrow \bb P(V' \oplus 1) \hookrightarrow \bb P(V \oplus 1) \hookrightarrow \bb P(W \oplus 1).$$ The obvious comparison gives that the statement in the lemma holds for a composition of inclusions if it holds for the other three pairs of inclusions.
By the above we can thus conclude by the splitting principle (cf. Section \ref{DRR1-section:splittingprinciple} in \cite{DRR1}).
\end{proof}
\end{proof}
\begin{Prop} \label{prop:reductiontolinebundles} Consider the "rough" excess isomorphism in Section \ref{section:roughexcess}, thus in particular the case in condition (c). Then that isomorphism coincides with the constructed isomorphism.
\end{Prop}
\begin{proof} We have already remarked in the definition that it is not difficult to see the isomorphism does not depend on the extension of the virtual bundle. By the last point of Proposition \ref{prop:deformsection} the Koszul resolution given in the construction of the "rough" excess isomorphism is trivialized on $D$ and $D'$, and deforms to the Koszul resolution in the model situation. By the argument of unicity for closed immersions, together with the established fact that the constructed excess isomorphism is stable under such transformations, we are done.
\end{proof}
\begin{Prop} Condition $(e)$ holds for the rougher excess isomorphism.
\end{Prop}
\begin{proof} It is not difficult to show that the trivializations interchange if and only if they interchange in the model situation obtained after a deformation to the normal cone. Denote $i: \bb P(N) \to \bb P(N \oplus 1)$ and $i': \bb P(N') \to \bb P(N' \oplus 1)$ the natural inclusions. Since $f^! i_! \calo_{\bb P(N)}$ is canonically trivialized, we can replace the trivialized virtual bundle with this one. Since $g^! i_! \calo_{\bb P(N\oplus 1)} \simeq {i'}_! \calo_{\bb P(N'\oplus 1)}$, by condition $(b)$ the upper square of the diagram of excess isomorphisms $$\xymatrix{ {f'}_! (\lambda_{-1}(E) \otimes {g'}^! (x \otimes f^! i_! \calo_{\bb P(N)})) \ar[r] \ar[d] &  g^! f_! (x \otimes f^! i_! \calo_{\bb P(N)}) \ar[d] \\
{f'}_! (\lambda_{-1}(E) \otimes {g'}^! (x)) \otimes {i'}_! \calo_{\bb P(N')} \ar[r] \ar[d] & g^! {f}_!(x) \otimes {i'}_! \calo_{\bb P(N')} \ar[d] \\
\lambda_{-1}({\xi'}^\vee) \otimes \lambda_{-1}({\pi'}^! E) \otimes {i'}_! \calo_{\bb P(N')} {g^! \pi^! x} \ar[r] \ar[d] & \ar[d] \lambda_{-1}(g^* \xi^\vee) \otimes {g^! \pi^! x} \otimes {i'}_! \calo_{\bb P(N')} \\
{i'}_! {i'}^! ((\lambda_{-1}({\xi'}^\vee) \otimes \lambda_{-1}({\pi'}^* E))\otimes {g^! \pi^! x}) \ar[r] & {i'}_! {i'}^! (\lambda_{-1}(g^* \xi^\vee) \otimes {g^! \pi^! x})}$$
commutes. Here $x$ is any virtual vector bundle, the lower middle square is by definition of the excess isomorphism in the model situation, and the lower square is obtained by applying the projection formula with respect to ${i'}$. We show that the trivializations are interchanged on the level of $\bb P(N')$. Since $\bb P(N)$ (resp. $\bb P(N')$) does not meet the image of $Y$ (resp. $X$) in $\bb P(N \oplus 1)$ (resp. $\bb P(N' \oplus 1)$), the sections determined by the given Koszul resolutions are everywhere non-vanishing. The statement now follows from the following claim: Suppose that
$$0 \to V' \to V \to V'' \to 0$$ is an exact sequence of vector bundles. Also suppose that ${V'}$ (and thus $V$) admits an everywhere non-vanishing section $\sigma$. Then $\lambda_{-1}({V'}^\vee)$ and $\lambda_{-1}({V}^\vee)$ are both canonically trivialized by the Koszul complex and these trivializations are interchanged by the isomorphism $$\lambda_{-1} (V^\vee) \simeq \lambda_{-1}({V'}^\vee) \otimes \lambda_{-1}({V''}^\vee).$$
For this claim, consider the following diagram where $V^\vee \to \calo$ and ${V'}^\vee \to \calo$ are given by the dual sections,
$$\xymatrix{
 & 0 \ar[d] & 0 \ar[d] & 0 \ar[d] &  \\
0 \ar[r] & {V''}^\vee \ar[d] \ar[r] & ({V}/\calo)^\vee \ar[d] \ar[r] & \ar[r] \ar[d] ({V'}/\calo)^\vee \ar[r] & 0 \\
0 \ar[r] & {V''}^\vee \ar[d] \ar[r] & V^\vee \ar[d] \ar[r] & \ar[r] {V'}^\vee \ar[d] & 0 \\
0 \ar[r] & 0 \ar[d] \ar[r] & \calo \ar[r] \ar[d] & \ar[r] \calo \ar[d] & 0 \\
 & 0  & 0  &  0 &  }$$
Here by assumption $(V/\calo)^\vee$ and $(V'/\calo)^\vee$ are vector bundles.
Compatibility with filtration gives a commutative diagram of isomorphisms
$$\xymatrix{ \lambda_{-1}(V^\vee) \ar[r] \ar[d]  & \ar[d] \lambda_{-1}(\calo) \otimes \lambda_{-1}((V/\calo)^\vee) \\
\lambda_{-1}({V'}^\vee) \otimes \lambda_{-1}({V''}^\vee) \ar[r] & \lambda_{-1}(\calo) \lambda_{-1}({V'}^\vee/\calo)\otimes \lambda_{-1}({V''}^\vee)}$$
where the rightmost isomorphism is given by $$\lambda_{-1}((V/\calo)^\vee) \simeq \lambda_{-1}((V'/\calo)^\vee)\otimes \lambda_{-1}((V'')^\vee)$$ tensored with the identity morphism between $\lambda_{-1}(\calo)$ and itself. But the trivialization of the two objects are given by $\lambda_{-1}(\calo) \simeq 0$, and are thus interchanged.
\end{proof}
\begin{Cor} The constructed isomorphism of functors is an isomorphism of determinant functors, and thus induces an isomorphism of additive functors.
\end{Cor}
\begin{proof} By Theorem \ref{Thm:deligne-virtual}, we have to verify that the properties of Definition \ref{def:determinantfunc} are interwoven in the above construction. The only non-obvious one is property (b). But this follows from property (a) and the above proposition.
\end{proof}
We thus conclude the demonstration in the case of a regular closed immersion.
\end{proof}
\section{Projective bundle projections - uniqueness and existence} \label{section:Projbundle}
The case of projective bundle projections is more elementary. We start with:
\begin{Prop} \label{Prop:tor-independentexcess} Suppose that the diagram in the theorem is Tor-independent. Then the excess isomorphism is necessarily that of Tor-independent base change.
\end{Prop}
\begin{proof} The excess bundle is trivial and we need to verify that the resulting excess isomorphism ${f'}_! {g'}^! \simeq g^! f_! $ is given by base change.  Consider the following cubical diagram: $$   
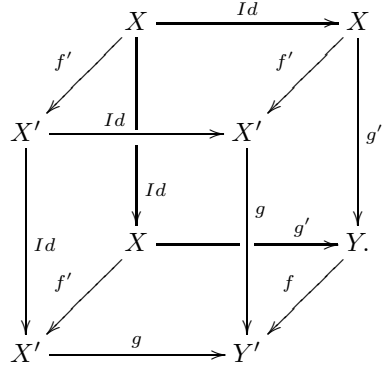
\xymatrix@!{
 &  X \ar@{->}[rr]^{Id} \ar@{->}'[d][dd]^{Id}
   & &  X \ar@{->}[dd]^{g'}
\\
 {X'} \ar@{<-}[ur]^{ f'} \ar@{->}[rr]^(0.40){ Id}\ar@{->}[dd]^{Id}
 & & {X'} \ar@{<-}[ur]^{f'}\ar@{->}[dd]^(0.35){g}
\\
 & X \ar@{->}'[r][rr]^{g' }
   & & Y.
\\
 X' \ar@{->}[rr]^{ g}\ar@{<-}[ur]^{f' }
 & & Y' \ar@{<-}[ur]^{f}
}$$
This satisfies the conditions of (a), and thus interchanges the excess isomorphism of the upper and lower via the base change isomorphism. The result follows if we can show that the excess isomorphism given by the upper diagram is the identity. By (d) and Lemma \ref{lemma:basechange} we only need to treat separately the case of $f'$ being a regular closed immersion and that of $f'$ being a projective bundle projection. The case of regular immersions easily follows from the previous section, using Proposition \ref{prop:reductiontolinebundles} and Lemma \ref{lemma:closed-base change}. For the projective bundle projection, we state it as a proposition:

\begin{Prop} \label{prop:projbundleproj}Suppose that $f$ is a projective bundle projection. Then the excess isomorphism associated to $f$ is uniquely determined and is necessarily given by Tor-independent base change.
\end{Prop}
\begin{proof} Consider the natural projection $f: \bb P(V) \to X$ for $V$ a vector bundle on $X$. We consider again the case when $g$ and $g'$ are the identity by the reduction above. Base change with $\bb P(V^\vee) \to X$ induces by Lemma \ref{lemma:projbundleformulastacks} an injection on the level of $K_1$ and thus a faithful functor on the level of virtual categories. We can furthermore suppose that $f: \bb P(V) \to X$ is such that $V$ has a canonical vector bundle $V' \subseteq V$ defining $\bb P(V') \to \bb P(V)$ as a relative divisor over $X$. The projective bundle formula again (cf. Lemma \ref{lemma:projbundleformulastacks}) shows that we only have to evaluate on objects of the form $f^! x \otimes \calo(i)$ for different $i$. Together with compatibility of compositions in (d), the case of regular closed immersions already considered and compatibility with the projection formula in (b), we see that the induced map is the identity on the direct image of $\bb P(V') \to X$ if and only if it is so for $f$. By induction we are reduced to the case when $f$ itself is the identity. But then compatibility with composition in (d) applied again shows that the induced automorphism on every object $x$ squares to itself, so it is necessarily the identity and we conclude.\end{proof}
We conclude the proposition.
\end{proof}
This forces us to define the excess isomorphism for projective bundles, and thus composition of projective bundles, as the base change isomorphism. Moreover, the base change isomorphism is compatible with base change and composition itself by Lemma \ref{lemma:basechange}. The trivialization condition is immediate, and condition (b) follows from a simple calculation of the cohomology sheaves considered above. We will return to its compatibility with the rough excess isomorphism in the next section.

\section{General excess isomorphism} \label{section:generalexcess}
In the rest of the paper we will tie together the isomorphisms constructed in the preceding sections and finally construct the total excess isomorphism. \\

Suppose first that we have a diagram $\mc E$ and a decomposition of a proper morphism $f: Y \to Y'$ as $f_\tau: Y \stackrel{i}{\rightarrow} \bb P_{Y'}(N) \stackrel{\pi}{\rightarrow} Y'$
where $i$ is a regular closed immersion,  $\pi$ is a projective bundle-projection and the subscript $\tau$ denotes this choice of factorization. By base change we obtain the composition of two Cartesian diagrams
$$\xymatrix{X \ar[d]^{i'} \ar[r]^{g''} & Y \ar[d]^i \\
\bb P_{X'}(g^* N) \ar[r]^{g'} \ar[d]^{\pi'} & \bb P_{Y'}(N) \ar[d]^\pi \\
X' \ar[r]^{g} & Y'.}$$
Denote by $E_\tau$ the associated excess bundle (of the upper square). In general, if we have two factorizations with a
morphism $r$,
$$\xymatrix{& P \ar[dd]^r \ar[dr]^\pi & \\
Y \ar[ur]^i \ar[dr]^{i'} & & Y' \\
& P' \ar[ur]^{\pi'} &  }$$ with $\pi$ and $\pi'$ smooth, we get an isomorphism
$$\phi_{\tau, \tau', r}: E_\tau \to E_{\tau'}$$
with the property that
$$\phi_{\tau', \tau'', r} \phi_{\tau, \tau', s} = \phi_{\tau, \tau'', sr}.$$
Now, given two arbitrary factorizations $\tau, \tau'$, we compare
them with the diagonal

$$\xymatrix{& & \bb P_{Y'}(N) \ar[dr] & \\
Y \ar[rr] \ar[urr] \ar[drr] & & \bb P_{Y'}(N) \times_{Y'} \bb P_{Y'}(M) \ar[u]^{\op{pr}_1} \ar[d]^{\op{pr}_2} \ar[r] & Y'\\
  &                   & \bb P_{Y'}(M) \ar[ur] }$$
and put
$$\phi_{\tau, \tau'} = \phi_{\tau, \tau \times \tau', \op{pr}_1} \left( \phi_{\tau', \tau \times \tau', \op{pr}_2}\right)^{-1}:E_{\tau'} \to
E_\tau.$$
This defines an isomorphism $\phi_{\tau, \tau'}: E_{\tau'} \to E_\tau$ which satisfies the cocycle condition $\phi_{\tau, \tau'} \phi_{\tau', \tau''} = \phi_{\tau, \tau''}$ so they glue together to an virtual excess-bundle $E$, determined up to canonical isomorphism. We define the excess isomorphism $\Psi_{\mc E, \tau}$ via the naive composition of the excess isomorphisms of the two diagrams. Since the isomorphism depends in a weak sense only on $f$ we will sometimes abbreviate it with $\Psi_{f, \tau}$. We give a list of basic compatibilities needed:

\begin{TL} \label{lemma:excessindep}We suppose the diagrams $\mc E$ (and thus the various morphisms $g$) implicit. \begin{itemize}
                                              \item Let $i$ be a section to a projection bundle-projection $\pi: \bb P_Y(N) \to Y$. Then $\Psi_\pi \Psi_i = \op{Id}$. \\
                                              \item Suppose that we have a Cartesian diagram $$\xymatrix{\bb P_Y(i^* N) \ar[d]^{\pi'} \ar[r]^{i'} & \bb P_{Y'}(N) \ar[d]^\pi \\
                                              Y \ar[r]^i & Y'}$$ with $i$ a regular immersion. Then $\Psi_i \Psi_{\pi'} = \Psi_{\pi} \Psi_{i'}$. \\
                                              \item Suppose we have a diagram $$\xymatrix{& \bb P_{Y'}(N)  \ar[d]^\pi \\
                                              Y \ar[ur]^j \ar[r]^i & Y'}$$ with $i$ being regular closed immersion and $\pi$ the projective bundle-projection. Then $\Psi_i = \Psi_\pi \Psi_j$.
                                            \end{itemize}
\end{TL}
\begin{proof} In the first case there is no excess and the statement becomes that the composition of two base change isomorphisms is the base change of the composition, which is the identity. \\
For the second case, we can suppose by additivity and standard reduction that our bundles are of the form $\pi^! F \otimes \calo(k)$ for $0 \leq k < n = \op{rk} N$.
The argument is now a lengthy but elementary application of the relationship $\pi_!  \pi^! = \op{Id}$ together with stability under base change already established for closed immersions, which we leave to the interested reader. \\
For the third point, notice that by \cite{SGA6}, VIII, Corollaire 1.2, $j$ (and the corresponding implicit $j'$) is also a regular closed immersion so the statement makes sense. One uses the Lichtenbaum-trick to reduce to the case of a morphism with section:
$$\xymatrix{ \bb P_Y(i^* N ) \ar@/^/[d]^{\pi'}  \ar[r]^{i'} & \bb P_{Y'}(N) \ar[d]^{\pi} \\
Y \ar[r]^i \ar@/^/[u]^s \ar[ur]^j & Y'.}$$
Here $s$ is the section determined by $j$ and base change. Then $\Psi_\pi \Psi_j = \Psi_\pi \Psi_{i'} \Psi_s = \Psi_i \Psi_{\pi'} \Psi_s = \Psi_i$ by the preceding statements.
\end{proof}

\begin{Prop} \label{prop:excessindep} With the above virtual excess bundle the $\Psi_{f, \tau}$ do not depend on choice of $\tau$.
\end{Prop}

\begin{proof} Let $\tau$ and $\tau'$ be two different factorizations as above. Considering again the diagram
$$\xymatrix{ & & \bb P_{Y'}(N) \ar[rrd]^\pi \\
Y \ar[urr]^i  \ar[rr]^{i''} \ar[rrd]^{i'} & & \bb P_{Y'}(N) \times_{Y'} \bb P_{Y'}(M) \ar[u]^{q'} \ar[d]^{q} \ar[rr]^{\pi''} & & Y' \\
& & \bb P_{Y'}(M) \ar[urr]^{\pi'}.}$$

By Proposition \ref{Prop:tor-independentexcess}, the excess isomorphism associated to $\pi''$ is necessarily given by Tor-independent base change and in particular defined. Put $\Psi_{f, \tau}$ equal to the naive composition of the
isomorphisms induced by $i$ and $\pi$. We obtain
equalities of isomorphisms
$$\Psi_{f, \tau'} = \Psi_\pi \Psi_i  = \Psi_{\pi} \Psi_{q'}
\Psi_{i''} = \Psi_{\pi''} \Psi_{i''} = \Psi_{\pi'} \Psi_{q} \Psi_{i''}  = \Psi_{\pi'} \Psi_{i'} =
\Psi_{f, \tau}.$$ Hence all isomorphisms $\Psi_{f, ?}$ are in fact
one single morphism, defining $\Psi_{\mc E}$.

\end{proof}

It remains to prove the following composition property, which is the only property which doesn't follow immediately from what is already established.

\begin{T} Suppose that we have morphisms
$$\xymatrix{X \ar[r]^f & Y \ar[r]^g & Z}$$ with $f$ and $g$ projective \lci morphisms.
 Then $$\Psi_f \circ \Psi_g = \Psi_{fg}.$$
\end{T}

We can factor as follows, for big enough $n$,

$$\xymatrix{
& & \bb P_Q(V) \ar[dr]^{t}  \ar[rr]^r &   & \bb P_Z(q_* (V^\vee \otimes \calo_{Q}(n))) \ar[dd]^P & \\
&  \bb P_Y(j^* V) \ar[ur]^{k} \ar[dr]^{p} & & \bb P_Z(V') = Q \ar[dr]^{q} & & \\
X \ar[ur]^{i} \ar[rr]^f & & Y \ar[rr]^g \ar[ur]^{j} & & Z \\}$$ where $V, V'$ and $q_* (V \otimes \calo_{Q} (n))$ are locally free (for large enough $n$). This can be written in this form because $Y$ has the resolution property and both $f$ and $g$ are supposed to be projective (see the argument in \cite{Riemann-RochAlgebra}, chapter IV, Proposition 3.12). Now, we have by Lemma \ref{lemma:excessindep}
and Proposition \ref{prop:excessindep}:
$$\Psi_g \Psi_{f} = \Psi_q \Psi_j \Psi_{p} \Psi_i = \Psi_{q} \Psi_t \Psi_k \Psi_i = \Psi_{P} \Psi_r \Psi_k \Psi_i = \Psi_P \Psi_{r k i} = \Psi_{g f}.$$
We conclude the proof of the main theorem since the necessary properties are preserved under the above composition.

\appendix
\section{Deformation to the normal cone } \label{appenix:deformation}

In this section we review for the convenience of the reader some very well-known facts about the deformation to the normal cone, without proof. A reference for details of the below is \cite{intersection}, chapter "Deformation to the Normal Cone" and also \cite{ComplexImmersions}, Section 4. \\
First of all, given a section $s$ of a rank $r$ vector bundle $E$ on an algebraic stack $X$, one has an induced dual section $s^\vee: E^\vee \to \calo_X$ given by the composition of $\calo_X \to E$ with $E \otimes E^\vee \to \calo_X.$ This leads to the Koszul complex $$0 \to \Lambda^r E^\vee \to  \Lambda^{r-1} E^\vee \to \ldots  \to \Lambda^2 E^\vee \to E^\vee \to \calo_X \to \calo_{Z(s)} \to 0,$$ which in local coordinates is given by $$e_1 \wedge \ldots \wedge e_j \mapsto \sum_{i = 0}^j (-1)^i s^\vee(e_i) e_1 \wedge \ldots \wedge \widehat{e_i} \wedge \ldots \wedge e_j$$ where $Z(s)$ is the zero-locus of the section $s$. We say that $s$ is a regular section if the above complex is exact and that the complex $\Lambda^\bullet E^\vee$ is a Koszul resolution of $\calo_{Z(s)}$. \\
Suppose that $i: X \to Y$ is a (representable) regular closed immersion of algebraic stacks, and let $N = N_{Y/X} = N_i$ be the normal bundle to $i$. We have a map $\bb P^1_X \to \bb P^1_Y$, and we define $M$ as the blow-up of $\bb P^1_Y$ in $X \times \{ \infty \}$ \textit{i.e.}  $\hbox{Proj}(\oplus_{n\geq 0} \mc I^n)$ where $\mc I$ is the ideal of the immersion $i$ (see \cite{MB}, 14.3). We have a natural morphism $\pi: M \to \bb P^1_Y \to \bb P^1_S$, so we have a diagram $$\xymatrix{\bb P^1_X \ar[rr] \ar[dr] & & M \ar[dl]^\pi \\ & \bb P^1}$$
since by the universal property of blowing up the map $\bb P^1_X \to \bb P^1_Y$ lifts to $M$. Moreover, $N_{X \times \{\infty \}/ \bb P^1_Y } \simeq N_{X/Y}  \oplus N_{\infty/\bb P^1} \simeq N \oplus 1$ (see \cite{intersection}, Appendix B.6.3 and chapter IV, Proposition 3.6). The last isomorphism $N_{\infty/\bb P^1} \simeq 1$ is non-canonical, but we fix one once and for all to lax the notation. The exceptional divisor on $M$ is a Cartier divisor isomorphic to $\bb P(N \oplus 1)$ and the fiber over $\infty$ is isomorphic to the union of this exceptional divisor and $Bl_X Y = \widetilde{Y}$, gluing together along a substack isomorphic to $\bb P(N)$
(here we put $\bb P(\mc F) = \hbox{Proj}(\hbox{Sym} \mc F^\vee )$ for a coherent sheaf $\mc F$). We have
that for $s \in \bb P^1$, $\pi^{-1}(s) = Y$ if $s \neq \infty$, and is equal to $\bb P(N \oplus 1) \cup  \widetilde{Y}$ for $s = \infty$, and the image
of $X$ does not meet $\widetilde{Y}$. In fact, it is embedded via the subbundle $1$ in $N \oplus 1$ where we embed it as $x \mapsto
(0,x)$. In other words, it embeds as $X = \bb P(1) \to \bb P(N \oplus 1)$ or as the zero-section $X \to N$ followed by the open
immersion $N \to \bb P(N \oplus 1)$. Let $p: \bb P(N \oplus 1) \to X$ be the projection. On $\bb P(N \oplus 1)$ we have a universal
exact sequence $$ 0 \to \calo_P(-1) \to p^* N \oplus 1 \to T_{P/X}(-1) \to 0 $$ where $\calo_P(-1)$ is the  universal sub bundle. By \cite{intersection}, Appendix B. 5.6, we know that the section determined by
\begin{equation} \label{sectionCanonical} \calo_P(-1) \to p^* N \oplus 1 \to p^* N \end{equation}
 is a regular section with zero-locus equal to $X = \bb P(1) \subseteq \bb P(N \oplus 1)$. We also have another section defined by $1 \to p^* N \oplus 1 \to T_{P/X}(-1)$ which is also regular with the same zero-locus. The reason we use the former in the constructions in the paper is due to the last point of the following proposition, to which we only indicate necessary references:

\begin{Prop} \label{prop:deformsection}

\begin{itemize}
    \item (\cite{Riemann-RochAlgebra}, chapter IV, Proposition 3.4) Suppose that we have two regular immersions $i: X \hookrightarrow Y, j: Y \hookrightarrow Z$ with normal bundles $N_i, N_j$ and $N_h$, where $h = j \circ i$. We have an exact sequence, localized on $X$:
    $$0 \to N_j \to N_h \to N_i \to 0.$$
    \item (\cite{intersection}, Appendix B, B.6.9) Applying the deformation to the normal cone to $i$ and
    $j \circ i$, the induced morphisms
    $$X \hookrightarrow \bb P(N_i \oplus 1) \hookrightarrow \bb P(N_h \oplus 1)$$
    are the natural ones.
    \item (\cite{intersection}, Appendix B, B.6.9) If we have a Cartesian diagram $$\xymatrix{X \ar[r]^{g'} \ar[d]^{f'} & \ar[d]^f Y \\
     X' \ar[r]^g &  Y' }$$
    where $f, f'$ are closed regular embeddings, the two deformations associated to them are
    compatible in the sense that over the infinite fiber they restrict to a Cartesian diagram
    $$\xymatrix{ X \ar[r]^{g'} \ar[d]^{f'} & \ar[d]^f Y \\
    \bb P(N_{f'} \oplus 1) \ar[r]^g & \bb P(N_f \oplus 1) }$$
    where the morphisms are the natural ones.
    \item (\cite{ComplexImmersions}, Theorem 4.8) Let $i: X \hookrightarrow Y$ be a closed regular embedding defined by a regular section of a vector bundle $N$ on $Y$ defining a Koszul resolution. Then there is a family of complexes of vector bundles parametrized by $\bb P^1$ on the total space of the deformation to the normal cone with the following properties: The restriction to fibers over $\bb P^1 \setminus \infty$ is the given Koszul resolution. The restriction to $\bb P(N_i \oplus 1)$ identifies with the Koszul-resolution $\pi^* N_{i}^\vee(-1) \to \calo_M$ given by (\ref{sectionCanonical}), and the restriction to $\widetilde{Y}$ is split acyclic.
\end{itemize}
\end{Prop}

\begin{TL}\label{lemma:coneterms} Let $$\xymatrix{X \ar[r]^{g'} \ar[d]^{f'} & Y \ar[d]^f \\
X' \ar[r]^g & Y' }$$ be a Cartesian diagram. Let $M'$ and $M$ be the
deformations to the normal cone of $X' \subseteq Y'$  and $X \subseteq Y$ respectively, and $G: M \to M'$ the induced map. Then we have
$$\bb P_{X'}(N' \oplus 1) = G^{-1}(\bb P_X(N \oplus
1)), \widetilde{Y'} = G^{-1}(\widetilde{Y})$$ and $$\bb P(N') = G^{-1}(\bb
P(N)).$$
\end{TL}
\begin{proof} Working on a smooth presentation, it follows from \cite{intersection}, Appendix B, B.6.9 that $G^{-1}$ maps the exceptional divisor to the exceptional divisor. Thus $G^{-1}(\bb P(N\oplus 1)) = \bb P(N'\oplus 1)$. Pulling back the fiber at infinity we obtain an equality of Cartier divisors $G^{-1}(\widetilde{Y} \cup \bb P(N\oplus 1)) = \widetilde{Y'} \cup \bb P(N'\oplus 1)$, so we necessarily have $G^{-1}(\widetilde{Y}) = \widetilde{Y'}$. The last point follows from the same reasoning, and the computation $\bb P(N') = \widetilde {Y'} \times_{\widetilde Y} \bb P(N) = M' \times_M \widetilde Y \times_{\widetilde Y} \bb P(N) = M' \times_M \bb P(N) = G^{-1}(\bb P(N)).$
\end{proof}
\providecommand{\bysame}{\leavevmode\hbox to3em{\hrulefill}\thinspace}
\providecommand{\MR}{\relax\ifhmode\unskip\space\fi MR }
\providecommand{\MRhref}[2]{%
  \href{http://www.ams.org/mathscinet-getitem?mr=#1}{#2}
}
\providecommand{\href}[2]{#2}

\end{document}